\documentclass[a4paper, 10pt]{article}\date{}\textheight=28.1cm \voffset=-4cm \textwidth=18.6cm \hoffset=-3.25cm
\usepackage{multirow,array,tabularx,amsthm,amsmath,amssymb,mathrsfs,mathtools,picture,graphicx,color,hyperref}
\hypersetup{colorlinks=true,linkcolor=red,citecolor=blue,filecolor=magenta,urlcolor=cyan}
\usepackage{tabularx,booktabs}

\allowdisplaybreaks[1]

\def\_email#1@#2\q_nil{\href{mailto:#1@#2}{{\emailfont #1\emailampersat #2}}}
\newcommand\emailampersat{{\color{cyan}\small@}}

\newtheorem{thm}{Theorem}[section]
\newtheorem{lma}[thm]{Lemma}
\newtheorem{prop}[thm]{Proposition}
\newtheorem{coro}[thm]{Corollary}
\newtheorem{rmk}[thm]{Remark}
\newtheorem{claim}[thm]{Claim}
\numberwithin{equation}{section}

\title{PB-steric equations: A general model of Poisson--Boltzmann equations}

\author{Jhih-Hong Lyu \thanks{Department of Mathematics, National Taiwan University, Taipei 10617, Taiwan (\tt d06221001@ntu.edu.tw).}, Tai-Chia Lin \thanks{Department of Mathematics, National Taiwan University, Taipei 10617, Taiwan; National Center for Theoretical Sciences, Mathematics Division, Taipei 10617, Taiwan ({\tt tclin@math.ntu.edu.tw}).}}

\begin{document}
\maketitle
\begin{abstract}When ions are crowded, the effect of steric repulsion between ions (which can produce oscillations in charge density profiles) becomes significant and the conventional Poisson--Boltzmann (PB) equation should be modified.
Several modified PB equations were developed but the associated total ionic charge density has no oscillation.
This motivates us to derive a general model of PB equations called the PB-steric equations with a parameter $\Lambda$, which not only include the conventional and modified PB equations but also have oscillatory total ionic charge density under different assumptions of steric effects and chemical potentials.
As $\Lambda=0$, the PB-steric equation becomes the conventional PB equation, but as $\Lambda>0$, the concentrations of ions and solvent molecules are determined by the Lambert type functions.
To approach the modified PB equations, we study the asymptotic limit of PB-steric equations with the Robin boundary condition as $\Lambda$ goes to infinity.
Our theoretical results show that the PB-steric equations (for $0\leq\Lambda\leq\infty$) may include the conventional and modified PB equations.
On the other hand, we use the PB-steric equations to find oscillatory total ionic charge density which cannot be obtained in the conventional and modified PB equations.
\vspace{5mm}
\vspace{1mm}
\end{abstract}

\section{Introduction}
\label{sec1}
Understanding the distribution of ions in the electrolyte is one of the most crucial problems in many physical and electrochemical fields.
As a well-known mathematical model, the Poisson--Boltzmann (PB) equations play a central role in the study of ionic distribution.
The conventional PB equations \cite{2003lamm,2008lu} treat ions as point particles without size and can be denoted as
\begin{align}
&-\nabla\cdot(\varepsilon\nabla\phi)=4\pi\rho_0+4\pi\sum_{i=1}^Nz_ie_0c_i(\phi),\label{eq:1.01}\\
&k_BT\ln(c_i)+z_ie_0\phi=\mu_i\quad\text{for}~i=1,\dots,N,\label{eq:1.02}
\end{align}
where $\phi$ is the electrostatic potential, $\varepsilon$ is the dielectric function, and $\rho_0$ is the permanent charge density function.
In addition, $N$ is the number of ion species, $c_i$ is the concentration of $i$th ion species with the valence $z_i\neq0$ for $i=1,\dots,N$.
Moreover, $c_i=c_i^{\text b}\exp(-z_ie_0\phi/k_BT)$, where $c_i^{\text b}$ is the concentration of the $i$th ion species in the bulk, $k_B$ is the Boltzmann constant, $T$ is the absolute temperature, $e_0$ is the elementary charge, and $\mu_i=k_BT\ln(c_i^{\text{b}})$ is the chemical potential.
However, when ions are crowded, steric repulsions may appear due to ion sizes, so the conventional PB equations should be modified (cf. \cite{2003eisenberg,2013eisenberg,1981klein,2020kohn}).

Under the hypothesis of volume exclusion, one may study the case of two-species ions with the same radius and obtain the following modified PB equations (cf. \cite{2011bazant,1942bikerman,1997borukhov,2017huang}):
\begin{align}
\label{eq:1.03}
&-\nabla\cdot(\varepsilon\nabla\phi)=4\pi(z_1e_0c_1+z_2e_0c_2),\\
\label{eq:1.04}
&c_i=c_i^{\text b}\frac{\exp(-z_ie_0\phi/k_BT)}{\displaystyle1-\gamma+\frac\gamma{z_1-z_2}(c_1^{\text b}\exp(-z_1e_0\phi/k_BT)+c_2^{\text b}\exp(-z_2e_0\phi/k_BT))},
\end{align}
for $i=1,2$, where $\gamma$ is the total bulk volume fraction of ions, and $z_1>0$ and $z_2<0$ are the valence of cations and anions, respectively.
Moreover, the concentrations of ion species in the bulk are given by $c_1^{\text b}=-z_2$ and $c_2^{\text b}=z_1$.
For the background and development of \eqref{eq:1.03}--\eqref{eq:1.04}, we refer the interested reader to \cite{2007grochowski,2007kornyshev}.
Moreover, when ions and solvent molecules may have different sizes, the associated PB equations become
\begin{align}
\label{eq:1.05}
&-\nabla\cdot(\varepsilon\nabla\phi)=4\pi\rho_0+4\pi\sum_{i=1}^Nz_ie_0c_i(\phi),\\
\label{eq:1.06}
&k_BT\ln(c_i)-k_BT\frac{v_i}{v_0}\ln(c_0)+z_ie_0\phi=\bar{\mu}_i\quad\text{for}~i=1,\dots,N,
\\
\label{eq:1.07}
&\sum_{i=0}^Nv_ic_i=1,
\end{align}
where $\displaystyle c_0=\frac1{v_0}(1-\sum\limits_{i=1}^Nv_ic_i)$ is the concentration of solvent molecules with the volume $v_0$ and the valence $z_0=0$, $v_i$ is the volume of $i$th ion species, and $\bar\mu_i$ is the associated chemical potential for $i=1,\dots,N$ (cf. \cite{2009li,2009li2,2013li,2011lu}).
Equations \eqref{eq:1.01}--\eqref{eq:1.07} can be denoted as
\begin{align}
&\label{eq:1.08}-\nabla\cdot(\varepsilon\nabla\phi)=4\pi\rho_0+4\pi\sum_{i=1}^Nz_ie_0c_i,\\&
\label{eq:1.09}
\mu_i=k_BT\ln(c_i)+z_ie_0\phi+\mu_i^{\text{ex}}+U_i^{\text{wall}}\quad\text{for}~i=0,1,\dots,N,
\end{align}
with different $\mu_i$ and $\mu_i^{\text{ex}}$, where $\mu_i$ is the chemical potential, $\mu_i^{\text{ex}}$ is the excess chemical potential which describes the interaction potential of ions and solvent molecules. Besides, $U_i^{\text{wall}}$ is the potential which comes from the interactions of ions and solvent molecules with the wall. Under $U_i^{\text{wall}}=0$ and conditions of $\mu_i$ and $\mu_i^{\text{ex}}$, equations \eqref{eq:1.08}--\eqref{eq:1.09} become the following equations.
\begin{itemize}
\item \eqref{eq:1.01}--\eqref{eq:1.02} if $\mu_i=k_BT\ln(c_i^\text b)$ and $\mu_i^{\text{ex}}=0$ for $i=1,\cdots,N$.
\item \eqref{eq:1.03}--\eqref{eq:1.04} if $\mu_i=k_BT\ln(c_i^\text b)$ and $\displaystyle\mu_i^{\text{ex}}=k_BT\ln\left(1-\gamma+\frac\gamma{z_1-z_2}(c_1^{\text b}e^{-z_1e_0\phi/k_BT}+c_2^{\text b}e^{-z_2e_0\phi/k_BT})\right)$ for $i=1,2$.
\item \eqref{eq:1.05}--\eqref{eq:1.07} if $\mu_i=\bar\mu_i$, $\mu_i^{\text{ex}}=-k_BT\frac{v_i}{v_0}\ln(c_0)$ for $i=1,\cdots,N$.
\end{itemize}

The steric interactions of ions can produce oscillations in concentration (density) profiles but such oscillations cannot be obtained by \eqref{eq:1.01}--\eqref{eq:1.04} (cf. \cite{2015gillespie,2010kalcher,2006NGjcp,1992TSDjcp}).
In \cite{2011zhou}, the oscillatory concentrations can be found in \eqref{eq:1.05}--\eqref{eq:1.07}, so it is natural to ask if there exists oscillatory total ionic charge density.
However, the total ionic charge density $\sum\limits_{i=1}^Nz_ic_i(\phi)$ of \eqref{eq:1.01}--\eqref{eq:1.07} is decreasing to $\phi$ (see \cite{2013li} and Propositions~\ref{prop:2.06} and~\ref{prop:2.10} below).
This motivates us to derive a general model of PB equations, called the Poisson--Boltzmann equations with steric effects (PB-steric equations), which not only include \eqref{eq:1.01}--\eqref{eq:1.07} but also have oscillatory total ionic charge density $\sum\limits_{i=1}^Nz_ic_i(\phi)$ under different assumptions of steric effects and chemical potentials.

To get the PB-steric equations, we consider to use the Lennard-Jones (LJ) potential which is a well-known model for the interaction between a pair of ions (and solvent molecules) and is important in the density functional theory and molecular dynamics simulation (cf. \cite{1993balbuena,1991kierlik,2012Lee,2015leimkuhler,1995rapaport}).
We set the energy functional $\iint_{\mathbb R^d}\psi(x-y)c_i(x)c_j(y)\,\mathrm dx\,\mathrm dy$ to describe the energy of steric repulsion of $c_i$ and $c_j$ for $i,j=0,1,\dots,N$.
Here $\psi(z)=|z|^{-12}$ comes from the repulsive part of LJ potential with strong singularity at the origin which makes the energy functional $\iint_{\mathbb R^d}\!\psi(x-y)c_i(x)c_j(y)\,\mathrm dx\,\mathrm dy$ hard to study.
Hence we replace the LJ potential by the approximate LJ potentials (cf. \cite{2012horng,2014lin}) to describe the steric repulsion of ions and solvent molecules and we obtain the following PB-steric equations:
\begin{align}
\label{eq:1.10}
&-\nabla\cdot(\varepsilon\nabla\phi)=4\pi\rho_0+4\pi\sum_{i=1}^Nz_ie_0c_i,
\\
\label{eq:1.11}&k_BT\ln(c_i)+z_ie_0\phi+\Lambda\sum_{j=0}^Ng_{ij}c_j=\Lambda\tilde\mu_i+\hat\mu_i\quad\text{for}~i=0,1,\dots,N,
\end{align}
where $c_0$ is the concentration of solvent molecules, $c_i$ is the concentration of $i$th ion species, and $\Lambda g_{ij}\geq0$ represents the strength of steric repulsion between the $i$th and $j$th ions (or solvent molecules).
One may see Appendix~\ref{Appendix.A} for the derivation of \eqref{eq:1.10}--\eqref{eq:1.11} when matrix $(g_{ij})$ is symmetric.

In this paper, we generalize \eqref{eq:1.10}--\eqref{eq:1.11} to all matrices $(g_{ij})$ (which may include nonsymmetric matrices) such that \eqref{eq:1.11} has a unique solution $c_i=c_i(\phi)$ for $\phi\in\mathbb R$ and $i=0,1,\dots,N$.
The nonsymmetry of matrix $(g_{ij})$ may come from the potential $U_i^{\text{wall}}$ which describes how the ions and solvent molecules interact with the wall.
Note that \eqref{eq:1.10}--\eqref{eq:1.11} can be expressed by \eqref{eq:1.08}--\eqref{eq:1.09} with the chemical potential $\mu_i=\Lambda\tilde\mu_i+\hat\mu_i$ and the excess chemical potential $\mu_i^{\text{ex}}$ and the potential $U_i^{\text{wall}}$ satisfying $\mu_i^{\text{ex}}+U_i^{\text{wall}}=\Lambda\sum\limits_{j=0}^Ng_{ij}c_j$ for $i=0,1,\dots,N$.
For the sake of simplify, we set $k_BT=e_0=1$ and write $4\pi\varepsilon$ instead of $\varepsilon$ so \eqref{eq:1.10}--\eqref{eq:1.11} become
\begin{align}
\label{eq:1.12}
&-\nabla\cdot\left(\varepsilon\nabla\phi\right)=\rho_0+\sum_{i=1}^Nz_ic_i\quad\text{in}~\Omega,\\
\label{eq:1.13}
&\ln(c_i)+z_i\phi+\Lambda\sum_{j=0}^Ng_{ij}c_j=\Lambda\tilde\mu_i+\hat\mu_i\quad\text{for}~i=0,1,\dots,N,
\end{align}
where $\Omega$ is a bounded smooth domain in $\mathbb R^d$ ($d\geq2$).
Under suitable conditions of $g_{ij}$ such that system \eqref{eq:1.13} has a unique solution $c_i=c_i(\phi)$ (which may depend on parameter $\Lambda$) for $\phi\in\mathbb R$ and $i=0,1,\dots,N$, \eqref{eq:1.12} becomes a nonlinear elliptic equation
\begin{equation}
\label{eq:1.14}
-\nabla\cdot(\varepsilon\nabla\phi)=\rho_0+\sum_{i=1}^Nz_ic_i(\phi)\quad\text{in}~\Omega.
\end{equation}
Then for all possible $g_{ij}\geq0$, equation \eqref{eq:1.14} and its limiting equation (as $\Lambda\to\infty$) are called as the PB-steric equations.
Hereafter, the boundary condition of \eqref{eq:1.14} is considered as the Robin boundary condition
\begin{equation}
\label{eq:1.15}
\phi+\eta\frac{\partial\phi}{\partial\nu}=\phi_{bd}\quad\text{on}~\partial\Omega,
\end{equation}
where $\phi_{bd}\in\mathcal C^2(\partial\Omega)$ is the extra electrostatic potential and $\eta\geq0$ is a constant related to the surface dielectric constant (cf. \cite{2005bazant,2007mori,2010ziebert}).

The strength of steric effect is determined by $\Lambda g_{ij}$'s in the PB-steric equations \eqref{eq:1.13}--\eqref{eq:1.14}.
As $\Lambda=0$ (or $g_{ij}=0$), \eqref{eq:1.13} has the same form as \eqref{eq:1.02} and the PB-steric equations \eqref{eq:1.13}--\eqref{eq:1.14} become the conventional PB equations \eqref{eq:1.01}--\eqref{eq:1.02}.
On the other hand, when $g_{ij}$'s are positive constants, a larger $\Lambda$ produces stronger steric repulsion. This motivates us to expect that as $\Lambda$ tends to infinity, the steric repulsion becomes extremely strong so that volume exclusion holds true and ions can be described by the lattice gas model as for the mean-field approximation of the modified PB equations \eqref{eq:1.03}--\eqref{eq:1.04} (cf. \cite{1997borukhov,2007grochowski}).
To justify this, we prove that as $\Lambda\to\infty$, the PB-steric equations \eqref{eq:1.13}--\eqref{eq:1.14} may approach to the modified PB equations \eqref{eq:1.03}--\eqref{eq:1.04} (see Theorem~\ref{thm:1.1} and Remark \ref{rmk1}).
Moreover, we prove that as $\Lambda\to\infty$, the PB-steric equations \eqref{eq:1.13}--\eqref{eq:1.14} may approach to the modified PB equations \eqref{eq:1.05}--\eqref{eq:1.07} (see Theorem~\ref{thm:1.5} and Remark~\ref{rmk:1.6}).

In order to obtain equations \eqref{eq:1.03}--\eqref{eq:1.04}, we need the assumptions of steric effects ($g_{ij}$) and chemical potentials ($\tilde\mu_i$) given by
\begin{itemize}
\item[(A1)] $g_{i0}=g_{00}=1-\gamma$ and $g_{ij}=g_{0j}=\gamma/Z$ for $i,j=1,\dots,N$,
\item[(A2)] $\tilde\mu_i=\tilde\mu_0$ for $i=1,\dots,N$,
\end{itemize}
where $0\leq\gamma\leq1$, $\tilde\mu_0>0$, $\hat\mu_0$ (in \eqref{eq:1.13}) and $Z=\sum\limits_{i=1}^N|z_i|$ are constants independent of $\Lambda$.
Then by (A1) and (A2), system \eqref{eq:1.13} has a unique solution $c_{i,\Lambda}=c_{i,\Lambda}(\phi)$ for $\phi\in\mathbb R$ and $i=0,1,\dots,N$, which satisfies
\begin{align}
\label{eq:1.16}&c_{i,\Lambda}=c_{0,\Lambda}\exp(\bar\mu_i-z_i\phi)\quad\text{for}~\phi\in\mathbb R~\text{and}~i=1,\dots,N,\\&
\label{eq:1.17}
\ln(c_{0,\Lambda})+\Lambda H(\gamma,z_1,\dots,z_N,\phi)c_{0,\Lambda}=\Lambda\tilde\mu_0+\hat\mu_0\quad\text{for}~\phi\in\mathbb R,\\&
\label{eq:1.18}H(\gamma,z_1,\dots,z_N,\phi)=1-\gamma+\frac\gamma Z\sum_{j=1}^Ne^{\bar\mu_j-z_j\phi},
\end{align}
where $\bar\mu_i=\hat\mu_i-\hat\mu_0$ is a constant independent of $\Lambda$ for $i=1,\dots,N$.
Note that \eqref{eq:1.17} can be solved in terms of the principal branch of Lambert W function (cf. \cite{2022mezo}), which implies that $c_{i,\Lambda}$ are positive smooth functions for $i=0,1,\dots,N$ (see Proposition~\ref{prop:2.01}).
Recall that the Lambert W function $W_0(x)$ can be defined by $W_0(x)e^{W_0(x)}=x$ for all $x\geq e^{-1}$, and the range of $W_0(x)$ is $[-1,\infty)$.
Hence the PB-steric equation \eqref{eq:1.14} can be expressed as
\begin{equation}
\label{eq:1.19}
-\nabla\cdot(\varepsilon\nabla\phi_\Lambda)=\rho_0+f_{\Lambda}(\phi_\Lambda)\quad\text{in}~\Omega,
\end{equation}
where function $f_\Lambda=f_\Lambda(\phi)$ is denoted as
\begin{equation}
\label{eq:1.20}
f_{\Lambda}(\phi)=\sum_{i=1}^Nz_ic_{i,\Lambda}(\phi)\quad\text{for}~\phi\in\mathbb R.
\end{equation}
By the standard method of a nonlinear elliptic equation, we may obtain the existence and uniqueness of \eqref{eq:1.19} with the Robin boundary condition \eqref{eq:1.15}.

Using the implicit function theorem on Banach spaces (cf. \cite[Theorem 15.1]{1973deimling}), we prove that $c_{i,\Lambda}$ converges to $c_i^*$ in space $\mathcal C^m[a,b]$ as $\Lambda$ tends to infinity for $m\in\mathbb N$ and $a<b$, where $c_i^*$ satisfies
\begin{equation}\label{eq:1.21}
c_i^*(\phi)=\frac{\tilde\mu_0e^{\bar\mu_i-z_i\phi}}{H(\gamma,z_1,\dots,z_N,\phi)}\quad\text{for}~\phi\in\mathbb R~\text{and}~i=0,1,\dots,N.
\end{equation}
Thus function $f_\Lambda$ also converges to $f^*$ in space $\mathcal C^m[a,b]$ as $\Lambda$ goes to infinity for $m\in\mathbb N$ and $a<b$, where
\begin{equation}
\label{eq:1.22}
f^*(\phi)=\sum_{i=1}^Nz_ic_i^*(\phi)\quad\text{for}~\phi\in\mathbb R\end{equation}
(see Corollary~\ref{coro:2.05}).
Moreover, for the asymptotic limit of \eqref{eq:1.19}, we obtain the following theorem.
\begin{thm}\label{thm:1.1}
Let $\Omega\subset\mathbb R^d$ be a bounded smooth domain, $\varepsilon\in\mathcal C^\infty(\overline\Omega)$ be a  positive function, $\rho_0\in\mathcal C^\infty(\overline\Omega)$, $\phi_{bd}\in\mathcal C^2(\partial\Omega)$, $z_0=0$, and $z_iz_j<0$ for some $i,j\in\{1,\dots,N\}$. Assume that (A1) and (A2) hold true. Then the solution $\phi_\Lambda$ of PB-steric equation \eqref{eq:1.19} with the Robin boundary condition \eqref{eq:1.15} satisfies
\[\lim_{\Lambda\to\infty}\|\phi_\Lambda-\phi^*\|_{\mathcal C^m(\overline\Omega)}=0\quad\text{for}~m\in\mathbb N,
\]
where $\phi^*$ is the solution of \begin{equation}
\label{eq:1.23}
-\nabla\cdot(\varepsilon\nabla\phi^*)=\rho_0+f^*(\phi^*)\quad\text{in}~\Omega
\end{equation} with the Robin boundary condition \eqref{eq:1.15}.
\end{thm}
\begin{rmk}\label{rmk1}
When $N=2$, \eqref{eq:1.21}--\eqref{eq:1.23} become \eqref{eq:1.03}--\eqref{eq:1.04}.
Besides, \eqref{eq:1.23} is the limiting equation of \eqref{eq:1.19} as $\Lambda\to\infty$.
This shows the PB-steric equations include the modified PB equations \eqref{eq:1.03}--\eqref{eq:1.04}.
\end{rmk}
\begin{rmk}
\label{rmk2}When $N=2$, we replace the assumption (A1) by
\begin{itemize}
\item[(A1)]\!\!$'$ $g_{i0}=g_{00}=1-\gamma$, $g_{ij}=g_{0j}=\gamma_j$ for $i,j=0,1,2$, and $\gamma=\gamma_1+\gamma_2$.
\end{itemize}
Then by (A1)\,$'$ and (A2), system \eqref{eq:1.13} has a unique solution $c_{i,\Lambda}=c_{i,\Lambda}(\phi)$ for $\phi\in\mathbb R$ and $i=0,1,2$.
Due to $z_1z_2<0$, we can calculate directly to get
$\mathrm df_\Lambda/\mathrm d\phi<0$, which implies function $f_\Lambda$ is decreasing to $\phi$.
Hence one can follow the similar argument in sections~\ref{sec:2} and \ref{sec:3} to obtain the same conclusion of Theorem~\ref{thm:1.1}.
\end{rmk}
\begin{rmk}When $N\geq2$, we replace the assumption (A1) by
\begin{itemize}
\item[(A1)]\!\!$''$ $g_{i0}=g_{00}=1-\gamma$, and $g_{ij}=g_{0j}=\gamma|z_j|/Z$ for $i,j=1,\dots,N$.
Then by (A1)\!$''$ and (A2), system \eqref{eq:1.13} has a unique solution $c_{i,\Lambda}=c_{i,\Lambda}(\phi)$ for $\phi\in\mathbb R$ and $i=0,1,\dots,N$.
Since $z_iz_j<0$ for some $i,j\in\{1,\dots,N\}$, we can calculate directly to get $\mathrm df_\Lambda/\mathrm d\phi<0$, which implies function $f_\Lambda$ is decreasing to $\phi$. 
Hence one can follow the similar argument in sections~\ref{sec:2} and \ref{sec:3} to obtain the same conclusion of Theorem~\ref{thm:1.1}.
\end{itemize}
\end{rmk}
\begin{rmk}
\label{rmk3}
As $N=3$, we replace the assumption (A1) by
\begin{itemize}
\item[(A1)]\!\!$'''$ $g_{i0}=g_{00}=1-\gamma$, $g_{ij}=g_{0j}=\gamma_j$ for $i,j=0,1,2,3$ and $\gamma=\gamma_1+\gamma_2+\gamma_3$.
\end{itemize}
Then by (A1)\,$'''$ and (A2), system \eqref{eq:1.13} has a unique solution $c_{i,\Lambda}=c_{i,\Lambda}(\phi)$ for $\phi\in\mathbb R$ and $i=0,1,2,3$.
But for some $\Lambda$, function $f_\Lambda$ and $f_\Lambda\circ\phi_\Lambda$ (total ionic charge density) may become oscillatory (see Figure~\ref{fig1}).
One can see the proof in Appendix~\ref{Appendix.B} and numerical methods in section~\ref{sec:4}.
Such oscillatory total ionic charge density $f_\Lambda=f_\Lambda(\phi)$ cannot be obtained in the conventional and modified PB equations \eqref{eq:1.01}--\eqref{eq:1.07}.
\begin{figure}\includegraphics[scale=0.6]{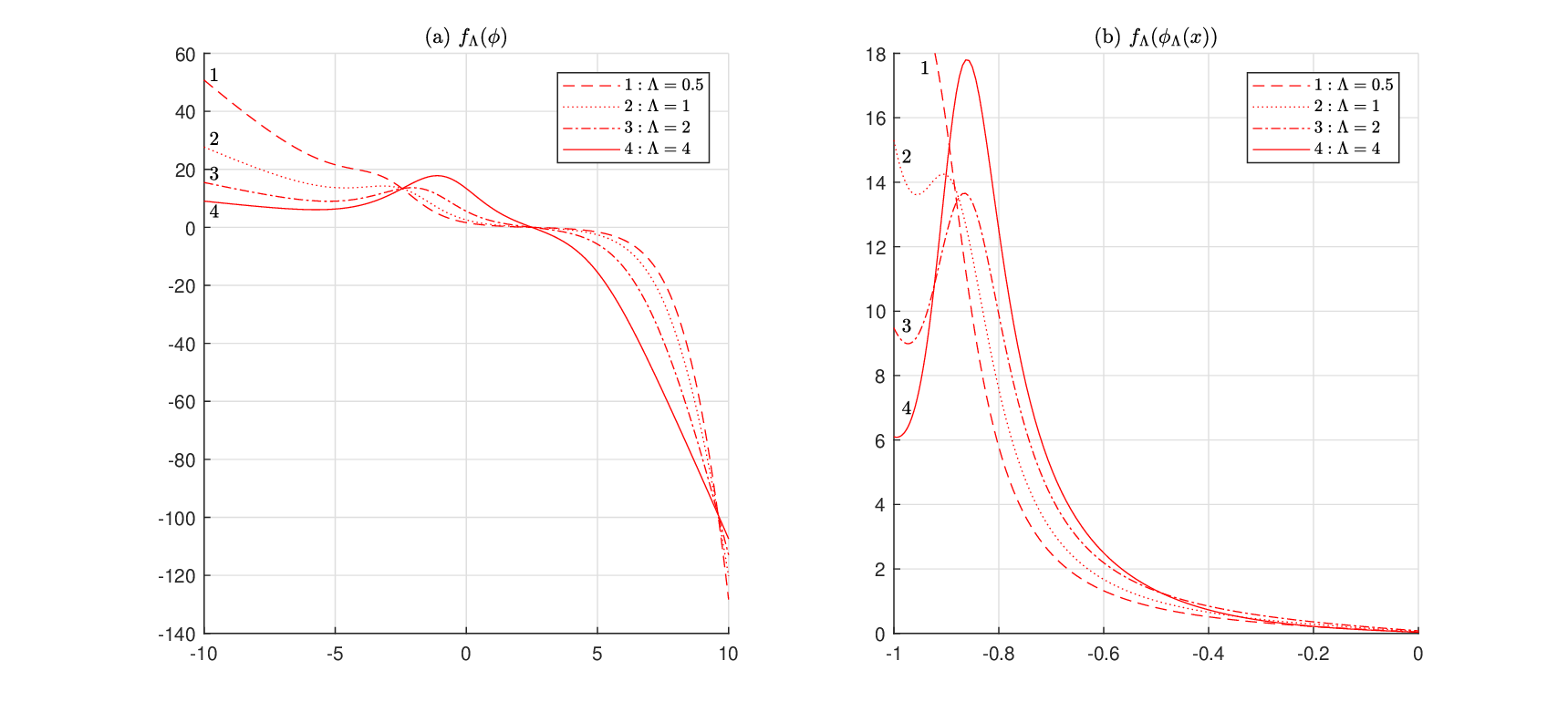}
\caption{The profiles of $f_\Lambda(\phi)$ and $f_\Lambda(\phi_\Lambda(x))$ under assumptions (A1)$'''$ and (A2).
In (a), curves 1--4 are profiles of function $f_\Lambda=\sum_{i=1}^3z_ic_{i,\Lambda}$ with $\Lambda=0.5$, $1$, $2$, $4$. In (b), curves 1--4 are the profiles of function $f_\Lambda\circ\phi_\Lambda$ with $\Lambda=0.5$, $1$, $2$, $4$.}
\label{fig1}
\end{figure}
\end{rmk}

To obtain \eqref{eq:1.05}--\eqref{eq:1.07}, we need the assumptions of $g_{ij}$ and $\tilde\mu_i$ given by
\begin{itemize}
\item[(A3)] $g_{ij}=\lambda_i\lambda_j$ for $i,j=0,1,\dots,N$,
\item[(A4)] $\tilde\mu_i=\lambda_i\tilde\mu_0$ for $i=0,1,\dots,N$,
\end{itemize}
where $\lambda_i>0$, $\tilde\mu_0>0$ and $\hat\mu_i$ (in \eqref{eq:1.13}) are constants independent of $\Lambda$.
Here $\tilde\mu_0$ is replaced by $\lambda_0\tilde\mu_0$ and $\lambda_0$ might not be $1$.
By (A3) and (A4), system \eqref{eq:1.13} has a unique solution $c_{i,\Lambda}=c_{i,\Lambda}(\phi)$ for $\phi\in\mathbb R$ and $i=0,1,\dots,N$, which satisfies
\begin{align}
\label{eq:1.24}
&c_{i,\Lambda}=(c_{0,\Lambda})^{\lambda_i/\lambda_0}e^{\bar\mu_i-z_i\phi}\quad\text{for}~\phi\in\mathbb R~\text{and}~i=1,\dots,N,\\&
\label{eq:1.25}
\ln(c_{0,\Lambda})+\Lambda\sum_{j=0}^N\lambda_0\lambda_j(c_{0,\Lambda})^{\lambda_j/\lambda_0}e^{\bar\mu_j-z_j\phi}=\Lambda\lambda_0\tilde\mu_0+\hat\mu_0\quad\text{for}~\phi\in\mathbb R,
\end{align}
where $\bar\mu_i=\hat\mu_i-\frac{\lambda_i}{\lambda_0}\hat\mu_0$ is a constant independent of $\Lambda$ for $i=0,1,\dots,N$.
Note that as $\lambda_j/\lambda_0=1$ for $j=1,\dots,N$, \eqref{eq:1.25} can be solved by the Lambert W function, but here some $\lambda_j/\lambda_0$ may not be equal to one so we may call $c_{0,\Lambda}$ as a Lambert-type function.
Then we apply the implicit function theorem on \eqref{eq:1.25} and obtain that $c_{0,\Lambda}(\phi)$ is a positive smooth function (see Proposition~\ref{prop:2.07}).
Hence the PB-steric equation \eqref{eq:1.14} can be expressed as
\begin{equation}
\label{eq:1.26}
-\nabla\cdot(\varepsilon\nabla\phi_\Lambda)=\rho_0+\tilde f_\Lambda(\phi_\Lambda)\quad\text{in}~\Omega,
\end{equation}
where function $\tilde f_\Lambda=\tilde f_\Lambda(\phi)$ is denoted as
\begin{equation}
\label{eq:1.27}
\tilde f_\Lambda(\phi)=\sum_{i=1}^Nz_ic_{i,\Lambda}(\phi)\quad\text{for}~\phi\in\mathbb R.\end{equation}
Note that the existence and uniqueness of \eqref{eq:1.26} with the Robin boundary condition \eqref{eq:1.15} can be obtain by the standard method of the nonlinear elliptic equation, but \eqref{eq:1.19} and \eqref{eq:1.26} are different because they have different nonlinear terms $f_\Lambda$ and $\tilde f_\Lambda$.

By the implicit function theorem on Banach spaces (cf. \cite[Theorem 15.1]{1973deimling}), we prove that $c_{i,\Lambda}$ converges to $c_i^*$ in space $\mathcal C^m[a,b]$ as $\Lambda$ tends to infinity for $m\in\mathbb N$ and $a<b$.
Here function $c_i^*$ satisfies
\begin{align}
\label{eq:1.28}
&c_i^*(\phi)=(c_0^*(\phi))^{\lambda_i/\lambda_0}e^{\bar\mu_i-z_i\phi}>0\quad\text{for}~\phi\in\mathbb R~\text{and}~i=1,\dots,N,\\&
\label{eq:1.29}
\sum_{i=0}^N\lambda_ic_i^*(\phi)=\sum_{i=0}^N\lambda_i(c_0^*(\phi))^{\lambda_i/\lambda_0}e^{\bar\mu_i-z_i\phi}=\tilde\mu_0\quad\text{for}~\phi\in\mathbb R,
\end{align}
where $\bar\mu_i=\hat\mu_i-\frac{\lambda_i}{\lambda_0}\hat\mu_0$ for $i=0,1,\dots,N$.
Thus function $\tilde f_\Lambda$ also converges to $\tilde f^*$ in space $\mathcal C^m[a,b]$ as $\Lambda$ goes to infinity for $m\in\mathbb N$ and $a<b$, where
\begin{equation}
\label{eq:1.30}\tilde f^*(\phi)=\sum_{i=1}^Nz_ic_i^*(\phi)\end{equation}
(see Corollary~\ref{coro:2.09}).
Moreover, for the asymptotic limit of \eqref{eq:1.26}, we have the following theorem.
\begin{thm}\label{thm:1.5}
Let $\Omega\subset\mathbb R^d$ be a bounded smooth domain, $\varepsilon\in\mathcal C^\infty(\overline\Omega)$ be a  positive function, $\rho_0\in\mathcal C^\infty(\overline\Omega)$, $\phi_{bd}\in\mathcal C^2(\partial\Omega)$, $z_0=0$, and $z_iz_j<0$ for some $i,j\in\{1,\dots,N\}$. Assume that (A3) and (A4) hold true. Then the solution $\phi_\Lambda$ of PB-steric equation \eqref{eq:1.26} with the Robin boundary condition \eqref{eq:1.15} satisfies
\[\lim_{\Lambda\to\infty}\|\phi_\Lambda-\phi^*\|_{\mathcal{C}^m(\overline{\Omega})}=0\quad\text{for}~m\in\mathbb N,
\]where $\phi^*$ is the solution of\begin{equation}
\label{eq:1.31}
-\nabla\cdot(\varepsilon\nabla\phi^*)=\rho_0+\tilde f^*(\phi^*)\quad\text{in}~\Omega
\end{equation} with the Robin boundary condition \eqref{eq:1.15}.
\end{thm}
\begin{rmk}
\label{rmk:1.6}
As $\lambda_i=v_i/v_0$, $\tilde\mu_0=1/v_0$, $c_i^*=c_i$ and $\phi^*=\phi$, \eqref{eq:1.28}--\eqref{eq:1.31} become \eqref{eq:1.05}--\eqref{eq:1.07} (up to scalar multiples).
Besides, \eqref{eq:1.31} is the limiting equation of \eqref{eq:1.26} as $\Lambda\to\infty$.
This shows the PB-steric equations include the modified PB equations \eqref{eq:1.05}--\eqref{eq:1.07}.
\end{rmk}

The rest of the paper is organized as follows. In section~\ref{sec:2}, we analyze of functions $f_\Lambda$, $f^*$, $\tilde f_\Lambda$ and $\tilde f^*$ under the assumptions (A1)--(A2) and (A3)--(A4), respectively.
The proof of Theorems~\ref{thm:1.1} and \ref{thm:1.5} are stated in section~\ref{sec:3}.
Numerical schemes are shown in section~\ref{sec:4}.

\section{Analysis of nonlinear terms under (A1)--(A2) and (A3)--(A4)}
\label{sec:2}

\subsection{Analysis of \texorpdfstring{$c_{i,\Lambda}$}{ciΛ} and \texorpdfstring{$f_\Lambda$}{fΛ} under (A1) and (A2)}
\label{sec:2.1}

In this section, we firstly use assumptions (A1) and (A2) and apply Gaussian elimination to solve system \eqref{eq:1.13} and obtain a unique positive smooth solution $c_{i,\Lambda}=c_{i,\Lambda}(\phi)$ for $\phi\in\mathbb R$ and $i=0,1,\dots,N$.
Then we establish the asymptotic behavior of $c_{i,\lambda}$ (as $\phi$ tends to infinity) and the strict decrease and unboundedness of $f_\Lambda=\sum\limits_{i=1}^Nz_ic_{i,\Lambda}$ (see Proposition~\ref{prop:2.01}).

For functions $c_{i,\Lambda}$ and $f_{\Lambda}$ under the assumptions (A1) and (A2), we have the following.
\begin{prop}
\label{prop:2.01}
Assume that $z_0=0$, (A1) and (A2) hold true.
Suppose $z_iz_j<0$ for some $i,j\in\{1,\dots,N\}$.
\begin{itemize}
\item[(i)] System \eqref{eq:1.13} can be solved uniquely by
\begin{equation}
\label{eq:2.01}
c_{i,\Lambda}(\phi)=\frac{W_0\left(\Lambda H(\gamma,z_1,\dots,z_N,\phi)e^{\Lambda\tilde\mu_0+\hat\mu_0}\right)}{\Lambda H(\gamma,z_1,\dots,z_N,\phi)}e^{\bar\mu_i-z_i\phi}
\end{equation}
for $\phi\in\mathbb R$ and $i=0,1,\dots,N$, where $W_0$ denotes the principal branch of Lambert W function (cf. \cite{2022mezo}).
\item[(ii)] Function $f_{\Lambda}=\sum\limits_{i=1}^Nz_ic_{i,\Lambda}$ is strictly decreasing on $\mathbb R$.
\item[(iii)] Let $I=\{i:z_i=\max\limits_{0\leq k\leq N}z_k\}$ and $J=\{j:z_j=\min\limits_{0\leq k\leq N}z_k\}$. Then
\begin{itemize}
\item[(a)] $\lim\limits_{\phi\to-\infty}c_{i_0,\Lambda}(\phi)=\infty$ for $i_0\in I$, and $\lim\limits_{\phi\to-\infty}c_{k,\Lambda}(\phi)=0$ for $k\notin I$;
\item[(b)] $\lim\limits_{\phi\to\infty}c_{j_0,\Lambda}(\phi)=\infty$ for $j_0\in J$, and $\lim\limits_{\phi\to\infty}c_{k,\Lambda}(\phi)=0$ for $k\notin J$.
\end{itemize}
\item[(iv)] For each $\Lambda>0$, the range of $f_\Lambda$ is entire space $\mathbb R$, and $\lim\limits_{\phi\to\pm\infty}f_\Lambda(\phi)=\mp\infty$.
\end{itemize}
\end{prop}
\begin{proof}By (A1) and (A2), system \eqref{eq:1.13} can be expressed as
\begin{equation}
\label{eq:2.02}
\ln(c_{i,\Lambda})+z_i\phi+\Lambda\left((1-\gamma)c_{0,\Lambda}+\frac\gamma Z\sum_{j=1}^Nc_{j,\Lambda}\right)=\Lambda\tilde\mu_0+\hat\mu_i\quad\text{for}~i=0,1,\dots,N.\end{equation}
Subtracting \eqref{eq:2.02} for $i\neq0$ from that for $i=0$, we get $\ln(c_{i,\Lambda}/c_{0,\Lambda})+z_i\phi=\hat\mu_i-\hat\mu_0:=\bar\mu_i$, which implies \eqref{eq:1.16} under the assumption $z_0=0$.
Then we plug \eqref{eq:1.16} into \eqref{eq:2.02} for $i=0$ to get \eqref{eq:1.17},
which implies
\begin{equation}
\label{eq:2.03}
\left[\Lambda H(\gamma,z_1,\dots,z_N,\phi)c_{0,\Lambda}\right]e^{\Lambda H(\gamma,z_1,\dots,z_N,\phi)c_{0,\Lambda}}=\Lambda H(\gamma,z_1,\dots,z_N,\phi)e^{\Lambda\tilde\mu_0+\hat\mu_0}\end{equation}
for $\phi\in\mathbb R$.
From theorems of Lambert W function in \cite{2022mezo}, \eqref{eq:2.03} has a unique smooth positive solution $c_{0,\Lambda}$ denoted by
\begin{equation}
\label{eq:2.04}
c_{0,\Lambda}(\phi)=\frac{W_0\left(\Lambda H(\gamma,z_1,\dots,z_N,\phi)e^{\Lambda\tilde\mu_0+\hat\mu_0}\right)}{\Lambda H(\gamma,z_1,\dots,z_N,\phi)}\quad\text{for}~\phi\in\mathbb R.
\end{equation}
By \eqref{eq:1.16} and \eqref{eq:2.04}, we arrive at \eqref{eq:2.01} and complete the proof of (i).

Next we state the proof of (ii).
We differentiate \eqref{eq:1.17} with respect to $\phi$ and get
\[\frac1{c_{0,\Lambda}}\frac{\mathrm dc_{0,\Lambda}}{\mathrm d\phi}+\Lambda H(\gamma,z_1,\dots,z_N,\phi)\frac{\mathrm dc_{0,\Lambda}}{\mathrm d\phi}-\Lambda c_{0,\Lambda}\frac\gamma Z\sum_{j=1}^Nz_je^{\bar\mu_j-z_j\phi}=0,\]
which gives
\begin{equation}\label{eq:2.05}
\frac{\mathrm dc_{0,\Lambda}}{\mathrm d\phi}=\frac{\displaystyle\frac\gamma Z\sum_{j=1}^Nz_je^{\bar\mu_j-z_j\phi}}{(\Lambda c_{0,\Lambda})^{-1}+H(\gamma,z_1,\dots,z_N,\phi)}c_{0,\Lambda}.\end{equation}
On the other hand, by \eqref{eq:1.16}, we write $f_\Lambda$ as $f_\Lambda(\phi)=\sum\limits_{i=1}^Nz_ic_{0,\Lambda}e^{\bar\mu_i-z_i\phi}$ for $\phi\in\mathbb R$.
Then differentiating $f_\Lambda$ with respect to $\phi$ and using \eqref{eq:2.05}, we get
\begin{equation}
\label{eq:2.06}
\frac{\mathrm df_{\Lambda}}{\mathrm d\phi}=\left(\frac{\displaystyle\frac{\gamma}{Z}\left(\sum_{i=1}^Nz_ie^{\bar\mu_i-z_i\phi}\right)\left(\sum_{i=1}^Nz_ie^{\bar\mu_i-z_i\phi}\right)}{(\Lambda c_{0,\Lambda})^{-1}+H(\gamma,z_1,\dots,z_N,\phi)}-\sum_{i=1}^Nz_i^2e^{\bar\mu_i-z_i\phi})\right)c_{0,\Lambda}.\end{equation}
Note the inequality
\begin{equation}
\label{eq:2.07}
\begin{aligned} &~~~\left(\sum_{i=1}^Nz_i^2e^{\bar\mu_i-z_i\phi}\right)\left[(\Lambda c_{0,\Lambda})^{-1}+H(\gamma,z_1,\dots,z_N,\phi)\right]>\frac\gamma Z\left(\sum_{i=1}^Nz_i^2e^{\bar\mu_i-z_i\phi}\right)\left(\sum_{i=1}^Ne^{\bar\mu_i-z_i\phi}\right)>\frac\gamma Z\left(\sum_{i=1}^Nz_ie^{\bar\mu_i-z_i\phi}\right)^2,\end{aligned}
\end{equation}
where we have used the fact that $z_iz_j<0$ for some $i,j\in\{1,\dots,N\}$.
Therefore, by \eqref{eq:2.06} and \eqref{eq:2.07}, we have $\mathrm df_{\Lambda}/\mathrm d\phi<0$ on $\mathbb R$ and the proof of (ii) is complete.

~

To prove (iii), we need the following claim.
\begin{claim}\label{claim:2.02}
There holds that
\begin{equation}
\label{eq:2.08}
\lim_{\phi\to\pm\infty}
\frac{W_0(\Lambda H(\gamma,z_1,\dots,z_N,\phi)e^{\Lambda\tilde\mu_0+\hat\mu_0})}{\ln\Lambda+\ln(H(\gamma,z_1,\dots,z_N,\phi))+\Lambda\tilde\mu_0+\hat\mu_0}=1.
\end{equation}
\end{claim}
\begin{proof}[Proof of Claim~\ref{claim:2.02}]
Recall the asymptotic behavior of Lambert W function in \cite{2008hoorfar,2022mezo}
\begin{equation}
\label{eq:2.09}
\ln x-\ln\ln x+\frac{\ln\ln x}{2\ln x}\leq W_0(x)\leq\ln x-\ln\ln x+\frac e{e-1}\frac{\ln\ln x}{\ln x}\quad\text{for}~x\geq e.
\end{equation}
Since $z_iz_j<0$ for some $i,j\in\{1,\dots,N\}$, it is clear that $\lim\limits_{\phi\to\pm\infty}\sum\limits_{i=1}^Ne^{\bar\mu_i-z_i\phi}=\infty$, which implies that\[\lim\limits_{\phi\to\pm\infty}\Lambda H(\gamma,z_1,\dots,z_N,\phi)e^{\Lambda\tilde\mu_0+\hat\mu_0}=\infty.\]
Then by \eqref{eq:2.09}, we obtain \eqref{eq:2.08} and complete the proof of Claim~\ref{claim:2.02}.
\end{proof}
Now we state the proof of (iii).
Since $z_iz_j<0$ for some $i,j\in\{1,\dots,N\}$, we know $z_{i_0}>0$ for $i_0\in I$ and $z_{j_0}<0$ for $j_0\in J$.
Then for $i_0\in I$, we may use \eqref{eq:1.16}, \eqref{eq:2.04} and Claim~\ref{claim:2.02} to get
\[\begin{aligned}\lim_{\phi\to-\infty}c_{i_0,\Lambda}(\phi)&=\lim_{\phi\to-\infty}\frac{e^{\bar\mu_{i_0}-z_{i_0}\phi}}{\Lambda H(\gamma,z_1,\dots,z_N,\phi)}W_0(\Lambda H(\gamma,z_1,\dots,z_N,\phi)e^{\Lambda\tilde\mu_0+\hat\mu_0})\\&=\frac Z{\Lambda\gamma\sum\limits_{i\in I}e^{\bar\mu_i-\bar\mu_{i_0}}}\lim_{\phi\to-\infty}W_0(\Lambda H(\gamma,z_1,\dots,z_N,\phi)e^{\Lambda\tilde\mu_0+\hat\mu_0})=\infty.\end{aligned}\]
On the other hand, for $k\notin I$, we have
\[\lim_{\phi\to-\infty}c_{k,\Lambda}(\phi)=\lim_{\phi\to-\infty}\frac{e^{\bar\mu_k-z_k\phi}[\left(\ln\Lambda+\ln(H(\gamma,z_1,\dots,z_N,\phi))+\Lambda\tilde\mu_0+\hat\mu_0\right)]}{\Lambda H(\gamma,z_1,\dots,z_N,\phi)}=0.\]
Thus the proof of (iii)(a) is complete.
The proof of (iii)(b) is similar to (iii)(a) by using \eqref{eq:1.16}, \eqref{eq:1.17}, and Claim~\ref{claim:2.02}. Therefore, we complete the proof of (iii).

Finally, we give the proof of (iv). By \eqref{eq:1.21}, function $f_\Lambda$ can be denoted as $\displaystyle f_\Lambda(\phi)=\sum_{i\in I}z_ic_{i,\Lambda}(\phi)+\sum_{i\notin I}z_ic_{i,\Lambda}(\phi)$ for $\phi\in\mathbb R$, where $I$ is defined in (iii).
Then by (iii), we have $\displaystyle\lim_{\phi\to-\infty}f_\Lambda(\phi)=\infty$ and $\displaystyle\lim_{\phi\to\infty}f_\Lambda(\phi)=-\infty$.
Therefore, we conclude (iv) and complete the proof of Proposition~\ref{prop:2.01}.
\end{proof}

\subsection{Analysis of \texorpdfstring{$c_i^*$}{c*} and \texorpdfstring{$f^*$}{f*} under (A1) and (A2)}
\label{sec:2.2}

Function $c_0^*$ is the limit $\lim\limits_{\Lambda\to\infty}c_{0,\Lambda}$, where function $c_{0,\Lambda}$ is the solution of \eqref{eq:1.17} for $\Lambda>0$.
Let $\delta=\Lambda^{-1}$ and $\tilde c_{0,\delta}=c_{0,\Lambda}$.
Then by \eqref{eq:1.17}, $\tilde c_{0,\delta}$ satisfies
\begin{equation}
\label{eq:2.10}
\delta\ln(\tilde c_{0,\delta}(\phi))+H(\gamma,z_1,\dots,z_N,\phi)\tilde c_{0,\delta}=\tilde\mu_0+\delta\hat\mu_0\quad\text{for}~\phi\in\mathbb R.
\end{equation}
Notice that $\Lambda\to\infty$ is equivalent to $\delta\to0^+$ so $c_0^*$ also equals the limit $\lim\limits_{\delta\to0^+}\tilde c_{0,\delta}$.
Moreover $c_0^*$ is defined by \eqref{eq:1.21}, which is \eqref{eq:2.10} with $\delta=0$.
The convergence of $\tilde c_{0,\delta}$ as $\delta\to0^+$, i.e. the convergence of $c_{0,\Lambda}$ as $\Lambda\to\infty$ is proved in Propositions~\ref{prop:2.03}, so by \eqref{eq:1.16}, we obtain the convergence of $c_{i,\Lambda}$ as $\Lambda\to\infty$.

\begin{prop}
\label{prop:2.03}
Let $c_i^*$ be defined in \eqref{eq:1.21}.
\begin{itemize}
\item[(i)] For $\phi\in\mathbb R$, $\lim\limits_{\Lambda\to\infty}c_{i,\Lambda}(\phi)=c_i^*(\phi)$ for $i=0,1,\dots,N$.
\item[(ii)] $\displaystyle\lim_{\Lambda\to\infty}\|c_{i,\Lambda}-c_i^*\|_{\mathcal C^m[a,b]}=0$ for $i=0,1,\dots,N$, $m\in\mathbb N$, and $a<b$, where $\|h\|_{\mathcal C^m[a,b]}:=\sum\limits_{k=0}^m\|h^{(k)}\|_\infty$ for $h\in\mathcal C^m[a,b]$.
\end{itemize}
\end{prop}
\begin{proof}To prove (i), we need the following claim.
\begin{claim}
\label{claim:2.04}Assume that (A1) and (A2) hold ture. Then there holds that
\[
\lim_{\Lambda\to\infty}\frac{W_0(\Lambda H(\gamma,z_1,\dots,z_N,\phi)e^{\Lambda\tilde\mu_0+\hat\mu_0})}{\Lambda\tilde\mu_0}=1\quad\text{for}~\phi\in\mathbb R.
\]
\end{claim}
\begin{proof}[Proof of Claim~\ref{claim:2.04}]The proof is similar to Claim~\ref{claim:2.02} so we omit it here.\end{proof}
By Proposition~\ref{prop:2.01}(i) and Claim~\ref{claim:2.04}, (i) follows from
\[
\lim_{\Lambda\to\infty}c_{i,\Lambda}(\phi)=e^{\bar\mu_i-z_i\phi}\lim_{\Lambda\to\infty}\frac{\Lambda\tilde\mu_0}{\Lambda H(\gamma,z_1,\dots,z_N,\phi)}=\frac{\tilde\mu_0e^{\bar\mu_i-z_i\phi}}{H(\gamma,z_1,\dots,z_N,\phi)}=c_i^*(\phi).\]

To prove (ii), we fix $m\in\mathbb N$, $a,b\in\mathbb R$ and $a<b$ arbitrarily.
Let $\|\cdot\|_{\mathcal C^m}:=\|\cdot\|_{\mathcal C^m[a,b]}$ for notational convenience.
For $\Lambda>0$, let $\delta=\Lambda^{-1}$, $\tilde c_{0,\delta}=c_{0,\Lambda}$, $w_\delta=\ln(\tilde c_{0,\delta})$, and $w^*=\ln(c_0^*)$.
Obviously, $\delta\to0^+$ is equivalent to $\Lambda\to\infty$.
Hence by \eqref{eq:1.16}, it suffices to show that $\lim\limits_{\delta\to0^+}\|e^{w_\delta}-e^{w^*}\|_{\mathcal C^m}=0$.
Because $w_\delta=\ln(\tilde c_{0,\delta})$ and $\tilde c_{0,\delta}=c_{0,\Lambda}$, \eqref{eq:2.10} can be denoted as $K_1(w_\delta(\phi),\delta)=0
$
for $\delta>0$ and $\phi\in[a,b]$, where $K_1$ is a $\mathcal C^1$-function on $\mathcal C^m[a,b]\times\mathbb R$ defined by
\begin{equation}
\label{eq:2.11}
K_1(w(\phi),\delta)=\delta w(\phi)+H(\gamma,z_1,\dots,z_N,\phi)e^{w(\phi)}-\tilde\mu_0-\delta\hat\mu_0
\end{equation}
for all $w\in\mathcal C^m[a,b]$ and $\phi\in[a,b]$.
Note that $K_1(w^*,0)=0$ by (i).
A direct calculation for Fr\'echet derivative of \eqref{eq:2.11} gives $D_wK_1(w(\phi),\delta)=\delta+H(\gamma,z_1,\dots,z_N,\phi)e^{w(\phi)}$ for all $w\in\mathcal C^m[a,b]$ and $\phi\in[a,b]$.
Then by (i), we get $D_wK_1(w^*(\phi),0)=\tilde\mu_0>0$ for all $\phi\in[a,b]$. This implies that $D_wK_1(w^*,0)I$ is a bounded and invertible linear map on the Banach space $\mathcal C^m[a,b]$, where $I$ is an identity map.
Hence by the implicit function theorem on Banach spaces (cf. \cite[Corollary 15.1]{1973deimling}), there exist an open subset $B_{\delta_0}(w^*)\times(-\delta_0,\delta_0)\subset\mathcal C^m[a,b]\times\mathbb R$ and a unique $\mathcal C^1$-function $\tilde w(\cdot,\delta)$ of $\delta\in(-\delta_0,\delta_0)$ with $\tilde w(\cdot,\delta)\in B_{\delta_0}(w^*)\subset\mathcal C^m[a,b]$ for $\delta\in(-\delta_0,\delta_0)$ such that $K_1(\tilde w(\cdot,\delta),\delta)=0$ for all $\delta\in(-\delta_0,\delta_0)$, which gives $\displaystyle\lim_{\delta\to0^+}\|\tilde w(\cdot,\delta)-w^*\|_{\mathcal C^m}=0$.
By Proposition~\ref{prop:2.01}(i), equation $K_1(w,\delta)=0$ has a unique solution $w=w_\delta$, which implies $\tilde w(\cdot,\delta)=w_\delta(\cdot)$ for $\delta\in(-\delta_0,\delta_0)$.
Therefore, we obtain $\displaystyle\lim_{\delta\to0^+}\|w_\delta-w^*\|_{\mathcal C^m}=0$, i.e., $\displaystyle\lim_{\delta\to0^+}\|e^{w_\delta}-e^{w^*}\|_{\mathcal C^m}=0$, and complete the proof of Proposition~\ref{prop:2.03}(ii).
\end{proof}

\begin{coro}\label{coro:2.05}
$\displaystyle\lim_{\Lambda\to\infty}\|f_\Lambda-f^*\|_{\mathcal C^m[a,b]}=0$ for $m\in\mathbb N$ and $a<b$.
\end{coro}
\begin{proof}
It follows from Proposition~\ref{prop:2.03}(ii) and the fact that $\displaystyle f_\Lambda(\phi)=\sum_{i=1}^Nz_ic_{i,\Lambda}(\phi)$ and $\displaystyle f^*(\phi)=\sum_{i=1}^Nz_ic_i^*(\phi)$ for $\phi\in\mathbb R$.
\end{proof}

For function $f^*$, we have the following proposition.

\begin{prop}Let $f^*$ be the function defined in \eqref{eq:1.22}.
\label{prop:2.06}
\begin{itemize}
\item[(i)] Function $f^*$ is strictly decreasing on $\mathbb R$.
\item[(ii)] Function $f^*$ satisfies $m^*<f^*(\phi)<M^*$ for all $\phi\in\mathbb R$, where\[m^*=\lim_{\phi\to\infty}f^*(\phi)<0~~\text{and}~~M^*=\lim_{\phi\to-\infty}f^*(\phi)>0.\]
\end{itemize}
\end{prop}
\begin{proof}
By \eqref{eq:1.22}, $f^*$ can be expressed as
\begin{equation}
\label{eq:2.12}f^*(\phi)=\frac{\tilde\mu_0}{H(\gamma,z_1,\dots,z_N,\phi)}\sum_{i=1}^Nz_ie^{\bar\mu_i-z_i\phi}\quad\text{for}~\phi\in\mathbb R.\end{equation}
Then differentiating \eqref{eq:2.12} with respect to $\phi$ gives
\[
\hspace{-1em}
\begin{aligned}
\frac{\mathrm df^*}{\mathrm d\phi}&=\tilde\mu_0\left(\frac{\displaystyle\frac\gamma Z\left(\sum_{i=1}^Nz_ie^{\bar\mu_i-z_i\phi}\right)^2}{\left(H(\gamma,z_1,\dots,z_N,\phi)\right)^2}-\frac{\sum_{i=1}^Nz_i^2e^{\bar\mu_i-z_i\phi}}{H(\gamma,z_1,\dots,z_N,\phi)}\right)\leq\frac{\tilde\mu_0\gamma}Z\frac{\displaystyle\left(\sum_{i=1}^Nz_ie^{\bar\mu_i-z_i\phi}\right)^2-\sum_{i=1}^Ne^{\bar\mu_i-z_i\phi}\sum_{i=1}^Nz_i^2e^{\bar\mu_i-z_i\phi}}{\left(H(\gamma,z_1,\dots,z_N,\phi)\right)^2}<0
\end{aligned}
\]
for $\phi\in\mathbb R$.
Here the last inequality comes from the fact that $z_iz_j<0$ for some $i,j\in\{1,\dots,N\}$. Therefore, we complete the proof of (i).

To prove (ii), we note that
\[m^*:=\tilde\mu_0\lim_{\phi\to\infty}\frac{\displaystyle\left(\sum_{j\in J}+\sum_{j\notin J}\right)z_je^{\bar\mu_j-z_j\phi}}{\displaystyle1-\gamma+\frac\gamma Z\left(\sum_{j\in J}+\sum_{j\notin J}\right)e^{\bar\mu_j-z_j\phi}}=\frac{\tilde\mu_0Z}\gamma\frac{\displaystyle\sum_{j\in J}z_je^{\bar\mu_j}}{\displaystyle\sum_{j\in J}e^{\bar\mu_j}}<0\]and
\[M^*:=\tilde\mu_0\lim_{\phi\to-\infty}\frac{\displaystyle\left(\sum_{i\in I}+\sum_{i\notin I}\right)z_ie^{\bar\mu_i-z_i\phi}}{\displaystyle1-\gamma+\frac\gamma Z\left(\sum_{i\in I}+\sum_{i\notin I}\right)e^{\bar\mu_i-z_i\phi}}=\frac{\tilde\mu_0Z}\gamma\frac{\displaystyle\sum_{i\in I}z_ie^{\bar\mu_i}}{\displaystyle\sum_{i\in I}e^{\bar\mu_i}}>0,\]
where $I$ and $J$ are defined in Proposition~\ref{prop:2.01}(iii).
By (i), $f^*$ is strictly decreasing so we have $m^*<f^*(\phi)<M^*$ for all $\phi\in\mathbb R$ and complete the proof of Proposition~\ref{prop:2.06}.
\end{proof}

\subsection{Analysis of \texorpdfstring{$c_{i,\Lambda}$}{ciΛ} and \texorpdfstring{$\tilde f_{\Lambda}$}{~fΛ} under (A3) and (A4)}
\label{sec:2.3}

In this section, we use assumptions (A3) and (A4) and apply Gaussian elimination to solve system \eqref{eq:1.13} and obtain a unique positive smooth solution $c_{i,\Lambda}=c_{i,\Lambda}(\phi)$ for $\phi\in\mathbb R$ and $i=0,1,\dots,N$.
Then we establish the asymptotic behavior of $c_{i,\Lambda}$ (as $\phi$ tends to infinity), and the strict decrease and unboundedness of $\displaystyle\tilde f_\Lambda=\sum_{i=1}^Nz_ic_{i,\Lambda}$ (see Proposition~\ref{prop:2.07}).

For functions $c_{i,\Lambda}$ and $\tilde f_\Lambda$, we have the following proposition.

\begin{prop}
\label{prop:2.07}
Assume that $z_0=0
$, (A3) and (A4) hold true.
Suppose $z_iz_j<0$ for some $i,j\in\{1,\dots,N\}$.
\begin{itemize}
\item[(i)] System \eqref{eq:1.13} has a unique, smooth, and positive solution $c_{i,\Lambda}=c_{i,\Lambda}(\phi)$ for $\phi\in\mathbb R$ and $i=0,1,\dots,N$.
\item[(ii)] Function $\tilde f_\Lambda=\sum\limits_{i=1}^Nz_ic_{i,\Lambda}$ is strictly decreasing on $\mathbb R$.
\item[(iii)] \begin{enumerate}
\item[(a)] If $z_k\geq0$ ($z_k\leq0$), then $\sup\limits_{\phi\geq0}c_{k,\Lambda}(\phi)\leq e^{\mu_k}$ ($\sup\limits_{\phi\leq0}c_{k,\Lambda}(\phi)\leq e^{\mu_k}$).
\item[(b)] There exist $i_0,j_0\in\{1,\dots,N\}$, $z_{i_0}>0$ and $z_{j_0}<0$ such that $\limsup\limits_{\phi\to-\infty}c_{i_0,\Lambda}(\phi)=\infty$ and $\limsup\limits_{\phi\to\infty}c_{j_0,\Lambda}(\phi)=\infty$.
\end{enumerate}
\item[(iv)] For each $\Lambda>0$, the range of $\tilde f_{\Lambda}$ is entire space $\mathbb R$, and $\displaystyle\lim_{\phi\to\pm\infty}\tilde f_\Lambda\left(\phi\right)=\mp\infty$.
\end{itemize}
\end{prop}\begin{proof}By (A3) and (A4), system \eqref{eq:1.13} can be expressed as
\begin{equation}
\label{eq:2.13}
\ln(c_{i,\Lambda})+z_i\phi+\Lambda\sum_{j=0}^N\lambda_i\lambda_jc_{j,\Lambda}=\Lambda\lambda_i\tilde\mu_0+\hat\mu_i\quad\text{for}~\phi\in\mathbb R~\text{and}~i=0,1,\dots,N.
\end{equation}
Following a similar process in the proof of Proposition~\ref{prop:2.01}(i), we can use \eqref{eq:2.13} to obtain \eqref{eq:1.24}--\eqref{eq:1.25}.
Then we denote \eqref{eq:1.25} as $k_1(c_{0,\Lambda},\phi)=0$, where $k_1$ is defined by
\[k_1(t,\phi)=\ln t+\Lambda\sum_{j=0}^N\lambda_0\lambda_j t^{\lambda_j/\lambda_0}\exp(\bar\mu_j-z_j\phi)-\Lambda\lambda_0\tilde\mu_0-\hat\mu_0\quad\text{for}~t>0~\text{and}~\phi\in\mathbb R.\]
Notice that, for any $\phi\in\mathbb R$, $k_1$ is strictly increasing for $t>0$ and the range of $k_1$ is entire space $\mathbb R$.
Then there exists a unique positive number $c_{0,\Lambda}(\phi)$ such that $k_1(c_{0,\Lambda}(\phi),\phi)=0$ for $\phi\in\mathbb R$.
Moreover, since $k_1$ is smooth for $t>0$, $\phi\in\mathbb R$, and
\[\frac{\partial k_1}{\partial t}(t,\phi)=\frac1t+\Lambda\sum_{j=0}^N\lambda_j^2t^{(\lambda_j-\lambda_0)/\lambda_0}\exp(\bar\mu_j-z_j\phi)>0\quad\text{for}~t>0~\text{and}~\phi\in\mathbb R,\]then by the implicit function theorem (cf. \cite[Theorem 3.3.1]{2002krantz}), $c_{0,\Lambda}=c_{0,\Lambda}(\phi)$ is a smooth and positive function on $\mathbb R$.
Therefore, by \eqref{eq:1.24}, each $c_{i,\Lambda}$ is smooth and positive on $\mathbb R$ and we complete the proof of (i).

To prove (ii), we differentiate \eqref{eq:2.13} with respect to $\phi$ and obtain
\begin{equation}
\label{eq:2.14}
(D_\Lambda+\Lambda G)\frac{\mathrm d{\bf c}_\Lambda}{\mathrm d\phi}=-{\bf z}\quad\text{for}~\phi\in\mathbb R,
\end{equation}
where $D_\Lambda=\text{diag}(\frac1{c_{0,\Lambda}},\cdots,\frac1{c_{N,\Lambda}})$ is positive definite, $G=[\lambda_0~\cdots~\lambda_N]^{\mathsf{T}}[\lambda_0~\cdots~\lambda_N]$, ${\bf c}_\Lambda=[c_{0,\Lambda}~\cdots~c_{N,\Lambda}]^{\mathsf{T}}$, and ${\bf z}=[z_0~\cdots~z_N]^{\mathsf{T}}$.
It is obvious that $D_\Lambda+\Lambda G$ is positive definite and invertible with inverse matrix $(D_\Lambda+\Lambda G)^{-1}$ which is also positive definite.
Then \eqref{eq:2.14} gives $\mathrm d{\bf c}_\Lambda/\mathrm d\phi=-(D_\Lambda+\Lambda G)^{-1}{\bf z}$, and $\mathrm d\tilde f_\Lambda/\mathrm d\phi$ becomes
\[
\frac{\mathrm d\tilde f_{\Lambda}}{\mathrm d\phi}=\sum_{i=1}^Nz_i\frac{\mathrm dc_{i,\Lambda}}{\mathrm d\phi}={\bf z}^{\mathsf T}\frac{\mathrm d{\bf c}_\Lambda}{\mathrm d\phi}=-{\bf z}^{\mathsf T}\left(D_\Lambda+\Lambda G\right)^{-1}{\bf z}<0\quad\text{for}~\phi\in\mathbb R.
\]
Here we have used the fact that ${\bf z}\neq0$ and we complete the proof of (ii).

To prove (iii), we firstly suppose that $z_k\geq0$ for some $k\in\{0,\dots,N\}$.
Then \eqref{eq:2.13} implies $\displaystyle\sup_{\phi\geq0}c_{k,\Lambda}(\phi)\leq e^{\mu_k}$.
Here we have used the fact that $c_{i,\Lambda}(\phi)>0$ for $\phi\in\mathbb R$ and $i=0,1,\cdots,N$.
The proof for the case of $z_k\leq0$ is similar to the case of $z_k\geq0$ so we omit it here.
Hence the proof of (iii)(a) is complete.
On the other hand, for (iii)(b), since $z_iz_j<0$ for some $i,j\in\{1,\dots,N\}$, then both index sets $I'=\{i:z_i>0\}$ and $J'=\{j:z_j<0\}$ are nonempty.
Now we claim that there exists $i_0\in I'$ such that $\displaystyle\limsup_{\phi\to-\infty}c_{i_0,\Lambda}(\phi)=\infty$.
We prove this by contradiction.
Suppose $\sup_{\phi\leq0}c_{i,\Lambda}(\phi)<\infty$ for all $i\in I'$.
Then there exists $K_1>0$ such that $0<c_{i,\Lambda}(\phi)<K_1$ for $\phi\leq0$ and $i\in I'$.
By equation \eqref{eq:2.13} and (iii)(a), we have
{\small
\[z_i\phi=\mu_i-\ln(c_{i,\Lambda}(\phi))-\Lambda\lambda_i\sum_{k=0}^N\lambda_jc_{k,\Lambda}(\phi)\geq\mu_i-\ln(K_1)-\Lambda\lambda_i\left(\sum_{k\in I'}\lambda_kK_1+\sum_{k\notin I'}\lambda_ke^{\mu_k}\right),\]}for $i\in I'$ and $\phi\leq0$,
which leads a contradiction by letting $\phi\to-\infty$.
Hence there exists $i_0\in I'$ such that $\displaystyle\limsup_{\phi\to-\infty}c_{i_0,\Lambda}(\phi)=\infty$.
The proof for the case that there exists $j_0\in J'$ such that $\displaystyle\limsup_{\phi\to\infty}c_{j_0,\Lambda}(\phi)=\infty$ is similar to the case that $\displaystyle\limsup_{\phi\to-\infty}c_{i_0,\Lambda}(\phi)=\infty$ for some $i_0\in I'$ so we omit it here. Therefore, the proof of (iii)(b) is complete.

The proof of (iv) is similar to the proof of Proposition~\ref{prop:2.01}(iv) so we omit it here. The proof of Proposition~\ref{prop:2.07} is complete.
\end{proof}

\subsection{Analysis of \texorpdfstring{$c_i^*$}{ci*} and \texorpdfstring{$\tilde f^*$}{~f*} under (A3) and (A4)}
\label{sec:2.4}

In this section, we apply the implicit function theorem to show the existence of $c_i^*$ and use the implicit function theorem on Banach space to show the convergence of $c_{0,\Lambda}$ as $\Lambda\to\infty$ (see Proposition~\ref{prop:2.08}). Then by \eqref{eq:1.24} and \eqref{eq:1.27}, we can obtain the convergence of $c_{i,\Lambda}$ and $f_\Lambda$ as $\Lambda\to\infty$.

\begin{prop}
\label{prop:2.08}Let $c_i^*$ be defined in \eqref{eq:1.28}--\eqref{eq:1.29}.
\begin{itemize}
\item[(i)] Equations \eqref{eq:1.28}--\eqref{eq:1.29} have a unique solution $(c_0^*,\dots,c_N^*)$ and each function $c_i^*=c_i^*(\phi)$ is a smooth and positive function for $\phi\in\mathbb R$ and $i=0,1,\dots,N$.
\item[(ii)] $\displaystyle\lim_{\Lambda\to\infty}\|c_{i,\Lambda}-c_i^*\|_{\mathcal C^m\left[a,b\right]}=0$ for $i=0,1,\dots,N$, $m\in\mathbb N$ and $a<b$.
\end{itemize}
\end{prop}
\begin{proof}
The proof of the proposition is similar to that of Proposition~\ref{prop:2.03} so we omit it here.
\end{proof}

\begin{coro}\label{coro:2.09}$\displaystyle\lim_{\Lambda\to\infty}\|\tilde f_\Lambda-\tilde f^*\|_{\mathcal C^m[a,b]}=0$ for $m\in\mathbb N$ and $a<b$.
\end{coro}
\begin{proof}
The proof of the corollary is similar to that of Corollary~\ref{coro:2.05} so we omit it here.
\end{proof}

For function $\tilde f^*$, we have the following propositions.
\begin{prop}
\label{prop:2.10}
Let $\tilde f^*$ be defined in \eqref{eq:1.30}.
\begin{itemize}
\item[(i)] Function $\tilde f^*$ is strictly decreasing on $\mathbb R$.
\item[(ii)] Function $\tilde f^*$ satisfies $m^*<\tilde f^*(\phi)<M^*$ for all $\phi\in\mathbb R$, where \[m^*=\lim_{\phi\to\infty}\tilde f^*(\phi)<0~\text{and}~M^*=\lim_{\phi\to-\infty}\tilde f^*(\phi)>0.\]
\end{itemize}
\end{prop}
\begin{proof}Differentiating \eqref{eq:1.28}--\eqref{eq:1.29} with respect to $\phi$ gives
\[
\sum_{j=0}^N\lambda_j\frac{\mathrm dc_j^*}{\mathrm d\phi}=0,\quad\frac{\mathrm dc_i^*}{\mathrm d\phi}=\frac{\lambda_i}{\lambda_0}\frac{c_i^*}{c_0^*}\frac{\mathrm dc_0^*}{\mathrm d\phi}-z_ic_i^*,\quad\text{for}~\phi\in\mathbb R~\text{and}~i=1,\dots,N,\]
which implies $\displaystyle
\frac{\mathrm dc_0^*}{\mathrm{d}\phi}=\lambda_0c_0^*\frac{\displaystyle\sum_{i=1}^Nz_i\lambda_ic_i^*}{\displaystyle\sum_{i=0}^N\lambda_i^2c_i^*}$ for $\phi\in\mathbb R$.
Consequently, we have
\[
\frac{\mathrm d\tilde f^*}{\mathrm d\phi}=\sum_{i=1}^Nz_i\frac{\mathrm dc_i^*}{\mathrm d\phi}
=\frac1{\lambda_0c_0^*}\frac{\mathrm dc_0^*}{\mathrm d\phi}\sum_{i=1}^Nz_i\lambda_ic_i^*-\sum_{i=1}^Nz_i^2c_i^*=\frac{\displaystyle\left(\sum_{i=1}^Nz_i\lambda_ic_i^*\right)^2}{\displaystyle\sum_{i=0}^N\lambda_i^2c_i^*}-\sum_{i=1}^Nz_i^2c_i^*<0\]
for $\phi\in\mathbb R$.
Here the last inequality comes from the Cauchy--Schwarz inequality and the fact that $\lambda_i>0$, $c_i^*>0$ for $i\in\{0,1,\dots,N\}$, and $z_iz_j<0$ for some $i,j\in\{1,\dots,N\}$.
Therefore, we complete the proof of (i).

To prove boundedness of $\tilde f^*$, we note that \eqref{eq:1.29} implies $
0<c_i^*(\phi)<\tilde{\mu}_0/\lambda_i$ for $\phi\in\mathbb R$ and $i=0,1,\dots,N$, and
\begin{equation}
\label{eq:2.15}
|\tilde f^*(\phi)|\leq\sum_{i=1}^N|z_i|c_i^*(\phi)\leq\tilde{\mu}_0\sum_{i=1}^N\frac{|z_i|}{\lambda_i}\leq\tilde{\mu}_0N\max_{0\leq i\leq N}\frac{|z_i|}{\lambda_i}\quad\text{for}~\phi\in\mathbb R.
\end{equation}
On the other hand, by \eqref{eq:1.29} and \eqref{eq:1.30}, function $\tilde f^*$ can be represented as
\begin{equation}
\label{eq:2.16}
\tilde f^*(\phi)=\sum_{i=1}^Nz_i(c_0^*(\phi))^{\lambda_i/\lambda_0}\exp(\bar{\mu}_i-z_i\phi)\quad\text{for}~\phi\in\mathbb R.
\end{equation}
By (i) and \eqref{eq:2.15}, function $\tilde f^*$ is strictly decreasing and bounded on $\mathbb R$ so the limit $\displaystyle\lim_{\phi\to\infty}\tilde f^*(\phi)$, denoted by $m^*$, exists and is finite.
Since $c_i^*<\tilde\mu_0/\lambda_i$ for $i=0,1\dots,N$, \eqref{eq:1.28} implies
\begin{equation}
\label{eq:2.17}
\lim_{\phi\to\infty}c_0^*(\phi)=\lim_{\phi\to\infty}\left[c_i^*(\phi)\exp(-\bar{\mu}_i+z_i\phi)\right]^{\lambda_0/\lambda_i}=0\quad\text{for}~z_i<0.
\end{equation}
Moreover, we use \eqref{eq:1.28} and \eqref{eq:2.17} to get
\begin{equation}
\label{eq:2.18}
\lim_{\phi\to\infty}c_i^*(\phi)=\lim_{\phi\to\infty}\left[(c_0^*(\phi))^{\lambda_i/\lambda_0}\exp(\bar{\mu}_i-z_i\phi)\right]=0\quad\text{for}~z_i>0,
\end{equation}
and hence \eqref{eq:2.16} and \eqref{eq:2.18} imply $\displaystyle m^*=\lim_{\phi\to\infty}\tilde f^*(\phi)\leq0$.
Now, we prove $m^*<0$ by contradiction.
Suppose $m^*=0$.
Then \eqref{eq:2.18} gives $\displaystyle\lim_{\phi\to\infty}c_i^*(\phi)=0$ for $i=0,1,\dots,N$, which contradicts with \eqref{eq:1.29} (by letting $\phi\to\infty$) and $\tilde\mu_0>0$.
Similarly, we may prove $\displaystyle\lim_{\phi\to-\infty}\tilde f^*(\phi)=M^*>0$ and complete the proof of Proposition~\ref{prop:2.10}.
\end{proof}

\section{Proof of Theorems~\ref{thm:1.1} and~\ref{thm:1.5}}
\label{sec:3}

Since $f_\Lambda$ and $\tilde f_\Lambda$ are unbounded on $\mathbb R$ but $f^*$ and $\tilde f^*$ are bounded on $\mathbb R$ so we cannot obtain the uniform convergence of $f_\Lambda$ and $\tilde f_\Lambda$ on $\mathbb R$ as $\Lambda$ goes to infinity (see Propositions~\ref{prop:2.01}, \ref{prop:2.06},~\ref{prop:2.07} and~\ref{prop:2.10}).
In order to use the convergence of $f_\Lambda$ and $\tilde f_\Lambda$ in space $\mathcal C^m[a,b]$ for $m\in\mathbb N$ and $a<b$, we firstly have to prove the uniform boundedness of $\phi_\Lambda$ (the solution of \eqref{eq:1.21} and \eqref{eq:1.26}) with respect to $\Lambda$ (see Lemmas~\ref{lma:3.1} and \ref{lma:3.4}).
Here we notice that $\rho_0$ may be any nonzero smooth function and the boundary condition may be the Robin boundary condition but not Dirichlet boundary condition so one cannot simply use the maximum principle on \eqref{eq:1.19} and \eqref{eq:1.20} to obtain the uniform boundedness of $\phi_\Lambda$.

\subsection{Uniform boundedness of \texorpdfstring{$\phi_\Lambda$}{} and \texorpdfstring{$c_{0,\Lambda}(\phi_\Lambda)$ under (A1) and (A2)}{}}
\label{sec:3.1}

\begin{lma}\label{lma:3.1}
There exist positive constants $M_1\geq1$, $M_2$, and $M_3$ independent of $\Lambda$ such that
\begin{itemize}
\item[(i)] $\displaystyle\max_{x\in\overline\Omega}c_{0,\Lambda}(\phi_\Lambda(x))\leq M_1$ for $\Lambda\geq1$.
\item[(ii)] $\|\phi_\Lambda\|_{L^\infty(\Omega)}\leq M_2$ for $\Lambda\geq1$.
\item[(iii)] $\displaystyle\min_{x\in\overline\Omega}c_{0,\Lambda}(\phi_\Lambda(x))\geq M_3$ for $\Lambda\geq1$.
\end{itemize}
\end{lma}
\begin{proof}Since $z_iz_j<0$ for some $i,j\in\{1,\dots,N\}$, then $\displaystyle\lim_{\phi\to\pm\infty}\sum_{i=1}^Ne^{\bar\mu_i-z_i\phi}=\infty$, and\[
K:=(\tilde\mu_0+\hat\mu_0)/\left(1-\gamma+\frac\gamma Z\min_{\phi\in\mathbb R}\sum_{i=1}^N\exp(\bar\mu_i-z_i\phi)\right)<\infty.\]
Let $M_1=\max\{K,1\}$ and $\Lambda\geq1$. We claim that $c_{0,\Lambda}(\phi_\Lambda(x))\leq M_1$.
Suppose that $\Omega_1=\{x\in\overline\Omega:c_{0,\Lambda}(\phi_\Lambda(x))>1\}$ is nonempty.
Otherwise, due to $M_1\geq1$, $c_{0,\Lambda}(\phi_\Lambda(x))\leq1\leq M_1$ for $x\in\overline\Omega$, which is trivial.
By \eqref{eq:1.17}, we have
$c_{0,\Lambda}(\phi_\Lambda(x))\leq\frac{\tilde\mu_0+\Lambda^{-1}\hat\mu_0}{H(\gamma,z_1,\dots,z_N,\phi_\Lambda(x))}\leq K\leq M_1$ for $x\in\Omega_1$.
Here we have used the fact that $\ln(c_{0,\Lambda}(\phi_\Lambda(x)))>0$ for $x\in\Omega_1$ and $\Lambda\geq1$.
Therefore, we complete the proof of (i).

To prove (ii), let $\psi$ be the solution of equation $-\nabla\cdot(\varepsilon\nabla\psi)=\rho_0$ in $\Omega$ with the Robin boundary condition $\psi+\eta\frac{\partial\psi}{\partial\nu}=0$ on $\partial\Omega$, and let $\bar\phi_\Lambda=\phi_\Lambda-\psi$.
Then function $\bar\phi_\Lambda$ satisfies
\begin{equation}
\label{eq:3.01}
-\nabla\cdot(\varepsilon\nabla\bar\phi_\Lambda)=f_\Lambda(\phi_\Lambda)\quad\text{in}~\Omega.
\end{equation}
Since $\psi$ is independent of $\Lambda$ and is continuous on $\overline\Omega$, then $\bar\phi_\Lambda$ is uniformly bounded if and only if $\phi_\Lambda$ is uniformly bounded, it suffices to show that $\displaystyle\max_{\overline\Omega}\bar\phi_\Lambda\leq M_2$ and $\displaystyle\min_{\overline\Omega}\bar\phi_\Lambda\geq-M_2$ for $\Lambda\geq1$, where $M_2$ is a positive constant independent of $\Lambda$.

Now we prove that $\displaystyle\max_{\overline\Omega}\bar\phi_\Lambda\leq M_2$ for $\Lambda\geq1$, where $M_2$ is a positive constant independent of $\Lambda$.
Suppose by contradiction that there exists a sequence $\Lambda_k$ with $\displaystyle\lim_{k\to\infty}\Lambda_k=\infty$ such that $\displaystyle\max_{\overline\Omega}\bar\phi_{\Lambda_k}\geq k$ for $k\in\mathbb N$.
Then there exists $x_k\in\Omega$ such that $\displaystyle\bar\phi_{\Lambda_k}(x_k)=\max_{\overline\Omega}\bar\phi_{\Lambda_k}$ which implies $\nabla\bar\phi_{\Lambda_k}\left(x_k\right)=0$ and $\Delta\bar\phi_{\Lambda_k}(x_k)\leq0$.
Note that because of the Robin boundary condition of $\bar\phi_{\Lambda_k}$, maximum point $x_k$ cannot be located on the boundary $\partial\Omega$ as $k$ sufficiently large.
Hence without loss of generality, we assume each $x_k\in\Omega$ for $k\in\mathbb N$.
For the sake of simplicity, we set $c_{i,k}:=c_{i,\Lambda_k}$, $\phi_k:=\phi_{\Lambda_k}$, $f_k:=f_{\Lambda_k}$, and $\bar\phi_k:=\bar\phi_{\Lambda_k}$.
Hence by \eqref{eq:3.01} with $\nabla\bar\phi_k(x_k)=0$, $\Delta\bar\phi_k(x_k)\leq0$ and function $\varepsilon$ is positive, we have
\begin{equation}
\label{eq:3.02}
0\leq-\nabla\varepsilon(x_k)\cdot\nabla\bar\phi_k(x_k)-\varepsilon(x_k)\Delta\bar\phi_k(x_k)=f_k(\phi_k(x_k)).
\end{equation}
Since $\displaystyle f_k(\phi_k)=\sum_{i=1}^Nz_ic_{i,k}(\phi_k)$, we can use \eqref{eq:3.02} to get
\begin{equation}
\label{eq:3.03}
0<\sum_{z_i<0}(-z_i)c_{i,k}(\psi(x_k)+\bar\phi_k(x_k))\leq\sum_{z_i>0}z_ic_{i,k}(\psi(x_k)+\bar\phi_k(x_k)).
\end{equation}
By Proposition~\ref{prop:2.01}(iii)(b) and \eqref{eq:3.03}, $\displaystyle\lim_{k\to\infty}\sum_{z_i<0}(-z_i)c_{i,k}(\psi(x_k)+\bar\phi_k(x_k))=0$.
This leads a contradiction with Proposition~\ref{prop:2.01}(iii)(b).
Hence we complete the proof to show that $\displaystyle\max_{\overline\Omega}\bar\phi_\Lambda\leq M_2$ for $\Lambda\geq1$, where $M_2$ is a positive constant independent of $\Lambda$.
The proof for the case that $\displaystyle\min_{\overline\Omega}\bar\phi_\Lambda\geq-M_2$ for $\Lambda\geq1$ is similar to the case that $\displaystyle\max_{\overline\Omega}\bar\phi_\Lambda\leq M_2$ for $\Lambda\geq1$ so we omit it here.
Thus, the proof of (ii) is complete.

Finally, we state the proof of (iii).
Suppose to the contrary that there exists $\Lambda_k$ with $\displaystyle\lim_{k\to\infty}\Lambda_k=\infty$, and $x_k\in\overline\Omega$ is the minimum point of $c_{0,\Lambda_k}(\phi_{\Lambda_k}(x))$ such that $\displaystyle\lim_{k\to\infty}c_{0,\Lambda_k}(\phi_{\Lambda_k}(x_k))=0$.
Notice that by \eqref{eq:1.17}, $c_{0,\Lambda}(\phi_\Lambda)>0$ for $\Lambda\geq1$ and $x\in\overline\Omega$.
Then there exists $N_1\in\mathbb N$ such that $\ln(c_{0,\Lambda_k}(\phi_{\Lambda_k}(x_k)))<0$ for $k>N_1$.
By \eqref{eq:1.17} with $\phi=\phi_{\Lambda_k}(x_k)$, we get
\begin{equation}
\label{eq:3.04}
H(\gamma,z_1,\dots,z_N,\phi_{\Lambda_k}(x_k))c_{0,\Lambda_k}(\phi_{\Lambda_k}(x_k))>\tilde\mu_0+\frac{\hat\mu_0}{\Lambda_k}\quad\text{for}~k>N_1.
\end{equation}
Taking $k\to\infty$, (ii) and \eqref{eq:3.04} imply $0\geq\tilde\mu_0$, which contradicts assumption (A2).
Therefore, we complete the proof of (iii). The proof of Lemma~\ref{lma:3.1} is complete.
\end{proof}

\begin{rmk}\label{rmk:3.2}
By Lemma~\ref{lma:3.1} and \eqref{eq:1.16}, we have $M_4\leq c_{i,\Lambda}(\phi_\Lambda(x))\leq M_5$ for $x\in\overline\Omega$, $\Lambda\geq1$ and $i=0,1,\dots,N$, where $M_4$ and $M_5$ are positive constants independent of $\Lambda$.
Moreover, due to $\displaystyle f_\Lambda(\phi_\Lambda)=\sum_{i=1}^Nz_ic_{i,\Lambda}(\phi_\Lambda)$, there exists a positive constant $M_6$ independent of $\Lambda$ such that $\|f_\Lambda(\phi_\Lambda)\|_{L^\infty(\overline\Omega)}\leq M_6$ for $\Lambda\geq1$.
\end{rmk}

\subsection{Convergence of \texorpdfstring{$\phi_\Lambda$}{phi-Lambda} under (A1) and (A2)}
\label{sec:3.2}

Since $\phi_\Lambda$ is the solution of equation $-\nabla\cdot(\varepsilon\nabla\phi_\Lambda)=\rho_0+f_\Lambda(\phi_\Lambda)$ in $\Omega$ with the Robin boundary condition $\phi_\Lambda+\eta\frac{\partial\phi_\Lambda}{\partial\nu}=\phi_{bd}$ on $\partial\Omega$, we use the $W^{2,p}$ estimate (cf. \cite[Theorem 15.2]{1959agmon}) to get
$\|\phi_\Lambda\|_{W^{2,p}(\Omega)}\leq C(\|\rho_0+f_\Lambda(\phi_\Lambda)\|_{L^p(\Omega)}+\|\phi_{bd}\|_{W^{1,p}(\Omega)})
$
for all $p>1$, where $C$ is a positive constant independent of $\Lambda$.
Hence by Remark~\ref{rmk:3.2}, we have the uniform bound estimate of $\phi_\Lambda$ in $W^{2,p}$ norm. This implies that there exists a sequence of functions $\{\phi_{\Lambda_k}\}_{k=1}^{\infty}$, with $\displaystyle\lim_{k\to\infty}\Lambda_k=\infty$, such that $\phi_{\Lambda_k}$ converges to $\phi^*$ weakly in $W^{2,p}(\Omega)$.
By Corollary~\ref{coro:2.05}, we have the convergence of function $f_{\Lambda_k}$ to $f^*$ in $\mathcal{C}^m[-M_2,M_2]$ for $m\in\mathbb N$, so $\phi^*$ satisfies \eqref{eq:1.20} in a weak sense, where the positive constant $M_2$ comes from Lemma~\ref{lma:3.1}(i).
Let $w_k=\phi_{\Lambda_k}-\phi^*$, $c_{i,k}:=c_{i,\Lambda_k}$, and $f_k:=f_{\Lambda_k}$.
Then by Sobolev embedding, $w_k\in\mathcal C^{1,\alpha}(\overline\Omega)$ for $\alpha\in(0,1)$, and $\displaystyle\lim_{k\to\infty}\|w_k\|_{\mathcal C^{1,\alpha}(\overline\Omega)}=0$.
Moreover, $w_k$ satisfies $
-\nabla\cdot(\varepsilon\nabla w_k)=f_k(w_k+\phi^*)-f^*(\phi^*)$ in $\Omega$ with the boundary condition $w_k+\eta\frac{\partial w_k}{\partial\nu}=0$ on $\partial\Omega$.
Using the Schauder's estimate (cf. \cite[Theorem 6.30]{1977gilbarg}) with the mathematical induction, we get
\begin{equation}
\label{eq:3.05}
\begin{aligned}
&{\color{white}\leq}\,\|w_k\|_{\mathcal C^{m+2,\alpha}(\overline\Omega)}
\leq C\|f_k(w_k+\phi^*)-f^*(\phi^*)\|_{\mathcal C^{m,\alpha}(\overline\Omega)}\\&\leq C'\sum_{i=1}^N\left(\|c_{i,k}(w_k+\phi^*)-c_i^*(w_k+\phi^*)\|_{\mathcal C^{m,\alpha}(\overline\Omega)}+\|c_i^*(w_k+\phi^*)-c_i^*(\phi^*)\|_{\mathcal C^{m,\alpha}(\overline\Omega)}\right),
\end{aligned}
\end{equation}
for all $m\in\mathbb N$ and $\alpha\in(0,1)$, where $C$ and $C'$ are positive constants independent of $k$.
By Proposition~\ref{prop:2.06}(i) and induction hypothesis $\displaystyle\lim_{k\to\infty}\|w_k\|_{\mathcal C^{m,\alpha}(\overline\Omega)}=0$, we may use \eqref{eq:3.05} to get $\displaystyle\lim_{k\to\infty}\|w_k\|_{\mathcal C^{m+2,\alpha}(\overline\Omega)}=0$, i.e. $\displaystyle\lim_{k\to\infty}\|\phi_{\Lambda_k}-\phi^*\|_{C^{m+2,\alpha}(\overline\Omega)}=0$ for $m\in\mathbb N$ and $\alpha\in (0,1)$.
Therefore, $\phi^*$ is the solution of \eqref{eq:1.26} with the Robin boundary condition \eqref{eq:1.15}.

To complete the proof of Theorem~\ref{thm:1.1}, we need to prove the following claim.
\begin{claim}
\label{claim:3.3}
For any $m\in\mathbb N$, we have $\displaystyle\lim_{\Lambda\to\infty}\|\phi_\Lambda-\phi^*\|_{\mathcal C^m(\overline\Omega)}=0$.
\end{claim}
\begin{proof}
Suppose that there exist the sequences $\{\Lambda_k\}$ and $\{\tilde{\Lambda}_k\}$ tending to infinity such that sequences $\{\phi_{\Lambda_k}\}$ and $\{\phi_{\tilde{\Lambda}_k}\}$ have limits $\phi_1^*$ and $\phi_2^*$, respectively.
It is clear that $\phi_1^*$ and $\phi_2^*$ satisfy \eqref{eq:1.31} with the Robin boundary condition \eqref{eq:1.15}.
Now we want to prove that $\phi_1^*\equiv\phi_2^*$.
Let $u=\phi_1^*-\phi_2^*$. Subtracting \eqref{eq:1.31} with $\phi^*=\phi_2^*$ from that with $\phi^*=\phi_1^*$, we obtain $-\nabla\cdot(\varepsilon\nabla u)=c_1(x)u$ in $\Omega$, where $c_1(x)=\frac{f^*(\phi_1^*(x))-f^*(\phi_2^*(x))}{\phi_1^*(x)-\phi_2^*(x)}$ if $\phi_1^*(x)\neq\phi_2^*(x)$; $\frac{\mathrm{d}f^*}{\mathrm{d}\phi}(\phi_1^*(x))$ if $\phi_1^*(x)=\phi_2^*(x)$.
By Proposition~\ref{prop:2.06}, we have $c_1<0$ in $\Omega$ which comes from the fact that if $f$ is a strictly decreasing function on $\mathbb R$, then $\frac{f(\alpha)-f(\beta)}{\alpha-\beta}<0$ for $\alpha\neq\beta$.
Since $\nabla\cdot(\varepsilon\nabla u)+c_1(x)u=0$ with $c_1<0$ in $\Omega$, it is obvious that $u$ cannot be a nonzero constant.
Then by the strong maximum principle, we have that $u$ attains its nonnegative maximum value and nonnpositive minimum value at the boundary point.
Suppose $u$ has nonnegative maximum value and attains its maximum value at $x^*\in\partial\Omega$.
Then by the boundary condition of $u$ which is $u+\eta\frac{\partial u}{\partial\nu}=0$ on $\partial\Omega$, we get $u(x^*)=-\eta\frac{\partial u}{\partial\nu}(x^*)\leq0$, and hence $u\leq u(x^*)\leq0$ on $\overline\Omega$.
Similarly, we obtain $u\geq0$ in $\overline\Omega$, and hence $u\equiv0$.
Therefore, we conclude that $\phi_1^*\equiv\phi_2^*$ and complete the proof of Theorem~\ref{thm:1.1}.
\end{proof}

\subsection{Proof of Theorem~\ref{thm:1.5}}
\label{sec:3.3}

\begin{lma}
\label{lma:3.4}
There exist positive constants $M_7\geq1$, $M_8$, and $M_9$ independent of $\Lambda$ such that
\begin{itemize}
\item[(i)] $\displaystyle\max_{x\in\overline{\Omega}}c_{0,\Lambda}(\phi_\Lambda(x))\leq M_7$ for $\Lambda\geq1$.
\item[(ii)] $\|\phi_\Lambda\|_{L^\infty(\Omega)}\leq M_8$ for $\Lambda\geq1$.
\item[(iii)] $\displaystyle\min_{x\in\overline{\Omega}}c_{0,\Lambda}(\phi_\Lambda(x))\geq M_9$ for $\Lambda\geq1$.
\end{itemize}
\end{lma}
\begin{proof}
The proof is identical to the proofs of Lemmas~3.1, 3.2, and 3.3 in \cite{2022lyu} so we omit it here.
\end{proof}

\begin{rmk}
\label{rmk:3.5}
By Lemma \ref{lma:3.4} and \eqref{eq:1.24}, we have $M_{10}\leq c_{i,\Lambda}(\phi_\Lambda(x))\leq M_{11}$ for $x\in\overline\Omega$, $\Lambda\geq1$ and $i=0,1,\dots,N$, where $M_{10}$ and $M_{11}$ are positive constants independent of $\Lambda$.
Moreover, due to $\displaystyle\tilde f_\Lambda(\phi_\Lambda)=\sum_{i=1}^N z_i c_{i,\Lambda}(\phi_\Lambda)$, $\|\tilde f_\Lambda(\phi_\Lambda)\|_{L^\infty(\overline\Omega)}\leq M_{12}$ for $\Lambda\geq1$, where $M_{12}$ is a positive constant independent of $\Lambda$.
\end{rmk}
For the proof of convergence of $\phi_\Lambda$ in $\mathcal C^m(\overline\Omega)$ for $m\in\mathbb N$, we can use Remark~\ref{rmk:3.5} and the same method in section~\ref{sec:3.2} to complete the proof of Theorem~\ref{thm:1.5}.

\section{Numerical methods}
\label{sec:4}

In this section, we introduce numerical methods to solve the PB-steric equations \eqref{eq:1.13}--\eqref{eq:1.14} with the Robin boundary condition \eqref{eq:1.15} to get the oscillation of total ionic charge densities $f_\Lambda$ for $\Lambda=0.5,1,2,4$ (see Figure~\ref{fig1}).
Throughout this section, we consider one dimensional domain $\Omega=(-1,1)$ and assume $N=3$, $z_0=0$, $z_1=1$, $z_2=-1$, $z_3=2$, $\tilde\mu_0=1$, $\hat\mu_0=0$, $\rho_0\equiv0$, $\eta=\varepsilon=0.1$, and $\phi_{bd}(\pm1)=\pm10$.
The assumptions (A1)$'''$ and (A2) are satisfied when the parameters $1-\gamma=\gamma_1=\gamma_2=0.01$, $\gamma_3=0.97$, $\bar\mu_0=\bar\mu_1=0$, and $\bar\mu_2=\bar\mu_3=-5$.

To solve \eqref{eq:1.13}, we use the command {\tt fsolve} in Matlab to perform the trust-region algorithm (cf. \cite{2011fan}) on the following discretized nonlinear equation:
\begin{align}
\label{eq:4.1}
\ln(\boldsymbol c_{i,\Lambda})+z_i\boldsymbol\phi+\Lambda\sum_{j=0}^Ng_{ij}\boldsymbol c_{j,\Lambda}=\Lambda\tilde\mu_i+\hat\mu_i\quad\text{for}~i=0,1,\dots,N,
\end{align}
where $\boldsymbol\phi=[\phi_0,\dots,\phi_L]^{\mathsf T}$ is the regular partition of the interval $[-10,10]$ with $L=1024$, and $\boldsymbol c_{i,\Lambda}=[c_{i,\Lambda}(\phi_0),\dots,c_{i,\Lambda}(\phi_L)]^{\mathsf T}$.
Moreover, in order to get the solution $\phi_\Lambda$ of \eqref{eq:1.13}--\eqref{eq:1.14}, we employ the Legendre--Gauss--Lobatto (LGL) points $\{x_k\}_{k=0}^L$ (cf. \cite{2015elbaghdady}) as the partition of the interval $[-1,1]$ with $L=256$ to discrtize \eqref{eq:1.13}--\eqref{eq:1.14} as the following algebraic equations:
\begin{align}
\label{eq:4.2}
&-\text D_{L+1}\boldsymbol\varepsilon\text D_{L+1}\boldsymbol\phi_\Lambda=\boldsymbol\rho_0+\sum_{i=1}^Nz_i\boldsymbol c_{i,\Lambda},\\&
\label{eq:4.3}
\ln(\boldsymbol c_{i,\Lambda})+z_i\boldsymbol\phi_\Lambda+\Lambda\sum_{j=0}^Ng_{ij}\boldsymbol c_{j,\Lambda}=\Lambda\tilde\mu_i+\hat\mu_i\quad\text{for}~i=0,1,\dots,N,
\end{align}
where $\text D_{L+1}=[d_{ij}]_{0\leq i,j\leq L}$ is the matrix satisfying $\text D_{L+1}\boldsymbol\psi=[\psi'(x_0),\dots,\psi'(x_L)]^{\mathsf T}$ for $\boldsymbol\psi=[\psi(x_0),\dots,\psi(x_L)]^{\mathsf T}$, $\boldsymbol\varepsilon
=[\varepsilon(x_0),\dots,\varepsilon(x_L)]^{\mathsf T}$, $\boldsymbol\rho_0=[\rho_0(x_0),\dots,\rho_0(x_L)]^{\mathsf T}$, $\boldsymbol\phi_\Lambda=[\phi_\Lambda(x_0),\dots,\phi_\Lambda(x_L)]^{\mathsf T}$, and $\boldsymbol c_{i,\Lambda}=[c_{i,\Lambda}(\phi_\Lambda(x_0)),\dots,c_{i,\Lambda}(\phi_\Lambda(x_L))]^{\mathsf T}$ for $i=0,1,\dots,N$.

For (A1)$'''$ and (A2), equations \eqref{eq:4.2} and \eqref{eq:4.3} can be denoted as
\begin{align}
\label{eq:4.4}
&-\text D_{L+1}\boldsymbol\varepsilon\text D_{L+1}\boldsymbol\phi_\Lambda=\boldsymbol\rho_0+\sum_{i=1}^Nz_i\boldsymbol c_{0,\Lambda}\circ\exp(\bar\mu_i-z_i\boldsymbol\phi_\Lambda),\\
\label{eq:4.5}
&\ln(\boldsymbol c_{0,\Lambda})+\Lambda\left(1-\gamma+\sum_{j=1}^N\gamma_j\exp(\bar\mu_j-z_j\boldsymbol\phi_\Lambda)\right)\circ\boldsymbol c_{0,\Lambda}=\Lambda\tilde\mu_0+\hat\mu_0.
\end{align}
In addition, we discretize the Robin boundary condition \eqref{eq:1.15} as
\begin{equation}
\label{eq:4.6}
\phi(x_0)-\eta\sum_{k=0}^Ld_{0k}\phi(x_k)=\phi_{bd}(-1),\quad\phi(x_L)+\eta\sum_{k=0}^Ld_{Lk}\phi(x_k)=\phi_{bd}(1),
\end{equation}
which replaces the first and last equations of \eqref{eq:4.4} and \eqref{eq:4.6}.
Then using the command {\tt fsolve} with the initial data $\boldsymbol\phi_\Lambda^{(0)}\equiv0$ and $\boldsymbol c_{0,\Lambda}^{(0)}\equiv1$, we obtain the profiles of $f_\Lambda(\phi)$ and $f_\Lambda(\phi_\Lambda(x))$ in Figure~\ref{fig1}.

\section{Conclusions}

In this paper, we derive the PB-steric equations with a parameter $\Lambda$, and prove that as $\Lambda\to\infty$, the solution of PB-steric equations \eqref{eq:1.13}--\eqref{eq:1.14} converges to the solution of modified PB equations \eqref{eq:1.03}--\eqref{eq:1.04} and \eqref{eq:1.05}--\eqref{eq:1.07} under different assumptions of steric effects and chemical potentials.
Equations \eqref{eq:1.01}--\eqref{eq:1.02}, \eqref{eq:1.03}--\eqref{eq:1.04}, and \eqref{eq:1.05}--\eqref{eq:1.07} were derived based on mean-field approximation (cf. \cite{1997borukhov,2009li}), whereas equations \eqref{eq:1.13}--\eqref{eq:1.14} are derived via adding the approximate Lennard--Jones potential.
We compare equations \eqref{eq:1.01}--\eqref{eq:1.07} and \eqref{eq:1.13}--\eqref{eq:1.14} in the following table to show that the PB-steric equations \eqref{eq:1.13}--\eqref{eq:1.14} can be regarded as a general model of PB equations.

\begin{center}\begin{tabular}{c|c|l}
Model&Total ionic charge density $\sum_{i=1}^Nz_ic_i(\phi)$& $c_i(\phi)$ ($i=0,1,\dots,N$)\\\hline
\multirow{2}{5.7em}{\eqref{eq:1.01}--\eqref{eq:1.02}}&\multirow{2}{5em}{decreasing}&$c_1$, $c_2$ are monotonic,\\&&$c_0\equiv0$ ($N=2$).\\
\hline
\multirow{2}{5.7em}{\eqref{eq:1.03}--\eqref{eq:1.04}}&\multirow{2}{5em}{decreasing}&$c_1$, $c_2$ are monotonic,\\&&$c_0\equiv0$ ($N=2$).\\
\hline
\multirow{3}{5.7em}{\eqref{eq:1.05}--\eqref{eq:1.07}}&\multirow{3}{5em}{decreasing}&Some of $c_i$'s are oscillatory;\\&&the others are monotonic.\\&&($N\geq2$)\\
\hline
\multirow{6}{5.7em}{\eqref{eq:1.13}--\eqref{eq:1.14}}&\multirow{3}{15em}{~~~~~~oscillatory (Remark~\ref{rmk3})}&Some of $c_i$'s are oscillatory;\\&&the others are monotonic.\\&&($N=3$)\\\cline{2-3}&\multirow{3}{17em}{decreasing (Propositions~\ref{prop:2.06} and \ref{prop:2.10})}&Some of $c_i$'s are oscillatory;\\&&the others are monotonic.\\&&($N\geq2$)
\end{tabular}
\end{center}

\appendix
\section{Derivation of \texorpdfstring{\eqref{eq:1.10}}{(1.10)} and \texorpdfstring{\eqref{eq:1.11}}{(1.11)}}
\label{Appendix.A}

In \cite{2014lin}, the approximate LJ potential $\Psi_{ij,\sigma}$ between the $i$th ion and the $j$th ion can be expressed as $\Psi_{ij,\sigma}(x)=(\Psi_{ij}\chi_\sigma\ast\varphi_\sigma)(x)$ for $x\in\mathbb R^d$ and $i,j=0,1,\dots,N$, where $\ast$ denotes the standard convolution, $\Psi_{ij}$ is the conventional LJ potential between $i$th and $j$th ions defined by $\displaystyle\Psi_{ij}(x)=\frac{\epsilon_{ij}(a_i+a_j)^{12}}{2|x|^{12}}$ for $x\in\mathbb R^d$ and $i,j=0,1,\dots,N$, and $\varphi_\sigma(x)$ is the spatially band-limited function defined by $\varphi_\sigma(x)=\mathcal F^{-1}(1-\chi_{\sigma^{-\gamma}}(\xi))$.
Here $\epsilon_{ij}$ denotes energy coupling constant between the $i$th and the $j$th ions, $a_0$ is the radius of solvent molecules, $a_i$ is the radius of the $i$th ion.
In addition, $\chi_\sigma$ and $\chi_{\sigma^{-\gamma}}$ are the characteristic function of the exterior ball $\mathbb R^d-\overline{B_\sigma(0)}$ and $\mathbb R^d-\overline{B_{\sigma^{-\gamma}}(0)}$ for $0<\gamma<1$, respectively, and $\mathcal F^{-1}$ is the inverse Fourier transform. More precisely, the characteristic function $\chi_s$ is defined by $\chi_s(z)=0~\text{if}~|z|\leq s; 1~\text{if}~|z|>s$. With the approximate LJ potential, we may define the approximate energy functional $E_{LJ,\sigma}$ and show that (cf. \cite{2014lin})
\begin{equation}
\label{eq:A.1}
E_{LJ,\sigma}[c_i,c_j]:=\Lambda g_{ij}\int_{\mathbb R^d}\!c_i(x)c_j(x)\,\mathrm dx\sim\int_{\mathbb R^d}\!c_i(x)(\Psi_{ij}\ast c_j)(x)\,\mathrm dx=E_{LJ}[c_i,c_j],\end{equation}
where $\Lambda=a^{12}\sigma^{d-12}$ satisfies $\displaystyle\lim_{\sigma\to0^+}\Lambda=\infty$, and $\displaystyle g_{ij}=\frac{\omega_d\epsilon_{ij}}{12-d}\frac{(a_i+a_j)^{12}}{a^{12}}$ denotes a dimensionless quantity.
Here $a=\min\{a_i:i=0,1\dots,N\}$ and $\omega_d$ is the surface area of a $d$ dimensional unit ball.
To derive \eqref{eq:1.10} and \eqref{eq:1.11}, we let energy functional
\[E[c_0,c_1,\dots,c_N,\phi]=E_{PB}[c_0,c_1,\dots,c_N,\phi]+\frac12\sum_{i,j=0}^NE_{LJ,\sigma}[c_i,c_j],\]
where $E_{LJ,\sigma}$ is defined in \eqref{eq:A.1} and
\[\begin{aligned}
&E_{PB}[c_0,c_1,\cdots,c_N,\phi]=\int_{\mathbb R^d}\left[-\frac12\varepsilon\left|\nabla\phi\right|^2+k_BT\sum_{i=0}^Nc_i(\ln c_i-1)+4\pi\left(\rho_0+\sum_{i=1}^Nz_ie_0c_i\right)\phi\right]\mathrm dx.\end{aligned}\]
Notice that $E_{PB}$ is the energy functional of conventional PB equation with the form $\displaystyle-\nabla\cdot(\varepsilon\nabla\phi)=\rho_0+\sum_{i=1}^Nz_ie^{\mu_i-z_i\phi}$ which can be obtained by $\delta E_{PB}/\delta\phi=0$ and $\delta E_{PB}/\delta c_i=\mu_i$ for $i=0,1,\dots,N$.
Here $\mu_i$ is the chemical potential.
Besides, $E_{LJ,\sigma}$ is the energy functional of the approximate LJ potentials (cf. \cite{2014lin}), and we may derive the equations
\begin{align}
&\label{eq:A.2}
-\nabla\cdot(\varepsilon\nabla\phi)=4\pi\rho_0+4\pi\sum_{i=1}^Nz_ie_0c_i,
\\\label{eq:A.3}
&
k_BT\ln(c_i)+z_ie_0\phi+\Lambda\sum_{j=0}^Ng_{ij}c_j=\mu_i\quad\text{for}~i=0,1,\dots,N,
\end{align}by
$\delta E/\delta\phi=0$ and $\delta E/\delta c_i=\mu_i$ for $i=0,1,\dots,N$ when matrix $(g_{ij})$ is symmetric.
From \eqref{eq:A.3}, constant $\mu_i$ can be expressed by\begin{equation}\label{eq:A.4}
\mu_i=\Lambda\tilde\mu_i+\hat\mu_i\quad\text{for}~i=0,1,\dots,N.\end{equation}
Therefore, by \eqref{eq:A.2}--\eqref{eq:A.4}, we have \eqref{eq:1.10}--\eqref{eq:1.11}.

\section{An example of oscillatory \texorpdfstring{$f_\Lambda$}{f-lambda} for Remark~\texorpdfstring{\ref{rmk3}}{3}}
\label{Appendix.B}

Under the assumptions (A1)$'''$ and (A2), system \eqref{eq:1.13} can be solved by
\[c_{i,\Lambda}(\phi)=\frac{\displaystyle W_0\left(\Lambda\left(1-\gamma+\sum_{j=1}^3\gamma_je^{\bar\mu_j-z_j\phi}\right)e^{\Lambda\tilde\mu_0+\hat\mu_0}\right)}{\displaystyle \Lambda\left(1-\gamma+\sum_{j=1}^3\gamma_je^{\bar\mu_j-z_j\phi}\right)}e^{\bar\mu_i-z_i\phi}\]
for $\phi\in\mathbb R$ and $i=0,1,2,3$.
Suppose that $\Lambda=4$, $z_0=0$, $z_1=-z_2=1$ $z_3=2$, $\bar\mu_0=\bar\mu_1=0$, $\bar\mu_2=\bar\mu_3=-5$, $\tilde\mu_0=1$, $\hat\mu_0=0$, $\gamma=0.99$, $\gamma_1=\gamma_2=0.01$, and $\gamma_3=0.97$.
Recall the derivative of the Lambert W function (cf. \cite{2008hoorfar,2022mezo}) is $\displaystyle\frac{\mathrm dW_0}{\mathrm dx}=\frac{W_0(x)}{x(1+W_0(x))}$ for $x>-1/e$.
Differentiating $c_{i,4}$ with respect to $\phi$, we obtain
\begin{equation}
\label{eq:B.1}
\begin{aligned}
\frac{\mathrm dc_{i,4}}{\mathrm d\phi}&=\frac{W_0(0.04e^4(1+e^{-\phi}+e^{-5+\phi}+97e^{-5-2\phi}))}{0.04(1+e^{-\phi}+e^{-5+\phi}+97e^{-5-2\phi})}\exp(\bar\mu_i-z_i\phi)\\&~~~\times\left(\frac{(e^{-\phi}-e^{-5+\phi}+194e^{-5-2\phi})W_0(0.04e^4(1+e^{-\phi}+e^{-5+\phi}+97e^{-5-2\phi}))}{(1+e^{-\phi}+e^{-5+\phi}+97e^{-5-2\phi})(1+W_0(0.04e^4(1+e^{-\phi}+e^{-5+\phi}+97e^{-5-2\phi})))}-z_i\right)
\end{aligned}
\end{equation}
for $\phi\in\mathbb R$ and $i=0,1,2,3$.
Let $K_1=0.04(1+e^2+e^{-7}+97e^{-1})$ and $K_2=e^2-e^{-7}+194e^{-1}$.
Then we have
\begin{equation}
\label{eq:B.2}
\frac{\mathrm df_4}{\mathrm d\phi}(-2)=\frac{W_0(e^4K_1)}{K_1}\left[\frac{K_2W_0(e^4K_1)}{25K_1(1+W_0(e^4K_1))}(e^2-e^{-7}+2e^{-1})-(e^2+e^{-7}+4e^{-1})\right].
\end{equation}
It is easy to check that $K_2/K_1>44$ and $3.3e^{3.3}<e^4K_1<3.4e^{3.4}$.
Due to the monotonic increase of $W_0$, we have $3.3<W_0(e^4K_1)<3.4$. Along with \eqref{eq:B.2}, we arrive at the fact $\displaystyle\frac{\mathrm df_4}{\mathrm d\phi}(-2)>0$.
On the other hand, from \eqref{eq:B.1}, we apply L'H\^ospital's law to get $\displaystyle\lim_{\phi\to\infty}\frac{\mathrm dc_{1,4}}{\mathrm d\phi}(\phi)=0$, $\displaystyle\lim_{\phi\to\infty}\frac{\mathrm dc_{2,4}}{\mathrm d\phi}(\phi)=25$, $\displaystyle\lim_{\phi\to\infty}\frac{\mathrm dc_{3,4}}{\mathrm d\phi}(\phi)=0$, which implies $\displaystyle\lim_{\phi\to\infty}\frac{\mathrm df_4}{\mathrm d\phi}(\phi)=-25$.
This shows $\displaystyle\frac{\mathrm df_4}{\mathrm d\phi}(\phi)<0$ when $\phi$ is sufficiently large.
Therefore, $f_4$ is oscillatory and the profile of $f_4$ is in Figure~\ref{fig1}(a).

\section*{Acknowledgments}
The authors express their thanks to reviewers for the careful and insightful review.

\bibliographystyle{siamplain}
%\bibliography{references}

\end{document}